\documentclass[leqno,10pt]{article}

\usepackage{amssymb, amsmath, amsthm, txfonts, mathrsfs}

\usepackage[all]{xy}

\theoremstyle{plain}
\newtheorem{theorem}{Theorem}[section]
\newtheorem*{theorem*}{Theorem}

\newtheorem{proposition}[theorem]{Proposition}

\newtheorem{lemma}[theorem]{Lemma}
\newtheorem{corollary}[theorem]{Corollary}

\theoremstyle{remark}

\theoremstyle{definition}
\newtheorem{definition}[theorem]{Definition}

\newcommand{\Att}{A(\!(t)\!)}
\newcommand{\cA}{\mathscr{A}}
\newcommand{\cCH}{\mathscr{C}\!H}
\newcommand{\Cf}{\textit{cf.}\;}
\newcommand{\CH}{\operatorname{CH}}
\newcommand{\Char}{\operatorname{char}}
\newcommand{\Coker}{\operatorname{Coker}}
\renewcommand{\d}{\operatorname{\partial}}
\newcommand{\dlog}{\operatorname{d\!\log}}
\newcommand{\dP}{\d_P}

\newcommand{\Ga}{\mathbb{G}_{a}}

\newcommand{\Gm}{\mathbb{G}_{m}}
\newcommand{\gh}{\operatorname{gh}}
\newcommand{\Gh}{\operatorname{Gh}}

\newcommand{\Id}{\operatorname{Id}}

\renewcommand{\Im}{\operatorname{Im}}
\newcommand{\inj}{\hookrightarrow}
\newcommand{\isomto}{\stackrel{\simeq}{\longrightarrow}}
\newcommand{\kbar}{\overline{k}}
\newcommand{\Ker}{\operatorname{Ker}}
\newcommand{\kP}{k(P)}
\newcommand{\kt}{k^{\times}}
\newcommand{\ktt}{k(\!(t)\!)}
\newcommand{\kPtt}{\kP(\!(t)\!)}
\newcommand{\Res}{\operatorname{Res}}
\newcommand{\KP}{K_P}
\newcommand{\Kt}{K^{\times}}
\newcommand{\KPt}{K_P^{\times}}

\newcommand{\llr}{\longleftrightarrow}
\newcommand{\M}{\mathscr{M}}
\renewcommand{\O}{\mathcal{O}}

\newcommand{\onto}[1]{\stackrel{#1}{\to}}
\newcommand{\onlong}[1]{\stackrel{#1}{\longrightarrow}}
\newcommand{\ol}[1]{\overline{#1}}
\newcommand{\Ohat}{\widehat{\O}}
\newcommand{\otimesM}{\!\stackrel{M}{\otimes}\!}
\newcommand{\otimesZ}{\otimes_{\Z}}

\newcommand{\Spec}{\operatorname{Spec}}
\newcommand{\surj}{\twoheadrightarrow}
\newcommand{\s}[1]{\{#1\}}

\newcommand{\TCH}{\operatorname{TCH}}
\newcommand{\Tr}{\operatorname{Tr}}

\newcommand{\ur}{\mathrm{ur}}
\newcommand{\wt}[1]{\widetilde{#1}}
\newcommand{\wh}[1]{\widehat{#1}}
\newcommand{\Wm}{\mathbb{W}_{m}}
\newcommand{\W}{\mathbb{W}}

\newcommand{\WmA}{\Wm(A)}

\newcommand{\WS}{\W_{S}}

\newcommand{\WSA}{\WS(A)}
\newcommand{\WSB}{\WS(B)}

\newcommand{\WSE}{\WS(E)}
\newcommand{\WSF}{\WS(F)}

\newcommand{\WSOmega}{\WS\Omega}
\newcommand{\WSOmegak}{\WS\Omega_k}
\newcommand{\WSOmegaE}{\WS\Omega_E}
\newcommand{\WSOmegaF}{\WS\Omega_F}

\newcommand{\WSOmegaEq}{\WSOmegaE^{q}}
\newcommand{\WSOmegaFq}{\WSOmegaF^{q}}
\newcommand{\WSOmegaEb}{\WSOmegaE^{\bullet}}
\newcommand{\WSOmegaFb}{\WSOmegaF^{\bullet}}

\newcommand{\Z}{\mathbb{Z}}

\def\sn{\smallskip\noindent}

\def\bn{\bigskip\noindent}
\def\ssm{\smallsetminus}

\title{An additive variant of Somekawa's $K$-groups and K\"ahler differentials}
\author{Toshiro Hiranouchi}

\begin{document}
\maketitle

\begin{abstract}
We introduce a Milnor type $K$-group 
associated to commutative algebraic groups over a perfect field. 
It is an additive variant of Somekawa's $K$-group. 
We show that the $K$-group associated to the additive group and 
$q$ multiplicative groups of a field is isomorphic to 
the space of absolute K\"ahler differentials of degree $q$ of the field, 
thus giving us a geometric interpretation 
of the space of absolute K\"ahler differentials. 
We also show that the $K$-group associated to 
the additive group and Jacobian variety of 
a curve is isomorphic to 
the homology group of a certain complex. 



\end{abstract}


\section{Introduction}

S.\ Bloch and H.\ Esnault introduced in \cite{BE03a} 
the additive higher Chow groups $\TCH^p(k,q) = \TCH^p(k,q;1)$  
of a field $k$ in an attempt 
to provide a cycle-theoretic description of the motivic cohomology 
theory of $X\otimes_k k[x]/(x^2)$ for a smooth $k$-scheme $X$.
They showed that when $q=p$, $\TCH^q(k,q) \simeq \Omega_k^{q-1}$ 
where $\Omega_k^{q-1}$ is the space of absolute K\"ahler differentials of $k$. 
This theorem is an analogue of the theorem 
of Nesterenko-Suslin/Totaro that $\CH^q(k,q) \simeq K_q^{M}(k)$, 
where $\CH^q(k,p)$ is the higher Chow group  
and $K_q^M(k)$ is the Milnor $K$-group. 
A reason that the space of K\"ahler differentials appears 
can be found in the calculation of $K_2(k[t]/(t^2))$ 
(\Cf \cite{BE03a}, Sect.\ 2). 
The aim of this note is to develop the analogy 
\[ 
    K_q^{M}(k)\quad \llr\quad \Omega_k^{q-1}
\]
by defining an additive variant of Somekawa's $K$-group \cite{Som90}, 
which corresponds loosely to the additive higher Chow groups above.  
The Milnor $K$-group $K_q^M(k)$ of a field $k$ is defined by the Steinberg relation: 
$\{x_1,\ldots ,x_q\} = 0$ if $x_i + x_j = 0\ (i\neq j)$. 
Although the definition of this relation is simple, 
but in general its geometric meaning is not at all clear. 
K.\ Kato and M.~Somekawa in \cite{Som90} considered to replace this Steinberg relation 
by the following classical theorems in Milnor $K$-theory ({\it e.g.,} \cite{BT73}). 

\sn
{\it The projection formula}:\  
For any finite field extension $k\subset E$, if $x_{i_0} \in E^{\times}$ 
and $x_i \in k^{\times}$ for all $i\neq i_0$, then 
\[
    N_{E/k}(\{x_1,\ldots , x_{i_0} ,\ldots, x_n\}) = \{x_1,\ldots , N_{E/k}(x_{i_0}),\ldots ,x_n\}
    \quad \mbox{in $K_q^M(k)$,}
\]
where $N_{E/k}$ is the norm map. 

\sn
{\it The Weil reciprocity law}:\ 
For the function field $k(C)$ of a projective smooth curve $C$ over $k$, the following sequence 
is a complex: 
\[
    K_{q+1}^M(k(C)) \xrightarrow{\oplus_P \d_P} 
    \bigoplus_{P\in C_0} K_q^M(k(P)) \xrightarrow{\sum_P N_{k(P)/k}} K_q^M(k),
\]
where 
$\d_P$ is the boundary map, $C_0$ is the set of closed points in $C$ and  
$k(P)$ is the residue field at $P$. 

\bn
We define 
\[
    K(k; \overbrace{\Gm,\ldots ,\Gm}^q) := 
    \left(\bigoplus_{E/k:\, \mathrm{finite}} \overbrace{E^{\times} \otimesZ \cdots \otimesZ E^{\times}}^q \right) /R. 
\]
Here, $R$ is the subgroup given by the relations 
of which mimic the statements of the theorems above (see \cite{Som90} for the precise definition). 
This group is universal with respect to these theorems, that is,    
the natural map 
\begin{equation}
    \label{eq:KS}
    K(k; \overbrace{\Gm,\ldots ,\Gm}^q) \to K_q^M(k) 
\end{equation}
annihilates the subgroup $R$ by the projection formula and the Weil reciprocity law,  
and one can show that the map is bijective (\cite{Som90}, Thm.\ 1.4). 
This group can be extended to semi-abelian varieties $G_1,\ldots ,G_q$ over $k$, 
the extension being denoted by $K(k;G_1,\ldots ,G_q)$. 
On the other hand, 
under the assumption that the field $k$ is perfect, 
the residue map  
\[
    \Res_P: \Omega_{k(C)}^{q} \to \Omega_{k}^{q-1}
\]
can take values in $\Omega_{k(P)}^{q-1}$ at a closed point $P$ of $C$, 
rather than $\Omega_k^{q-1}$, 
just as the boundary map on the Milnor $K$-group 
$\dP:K_{q+1}^M(k(C)) \to K_q^M(\kP)$ does. 
From the residue theorem (Thm.\ \ref{thm:res}), 
it follows that 
the refined residue map $\dP$ and the trace maps induce a complex   
\begin{equation}
    \label{eq:rt}
    \Omega_{k(C)}^q \xrightarrow{\oplus_P \d_P} \bigoplus_{P\in C_0} \Omega_{k(P)}^{q-1} 
    \xrightarrow{\sum_P \Tr_{k(P)/k}} \Omega^{q-1}_{k}. 
\end{equation}
We consider this as an additive version of the Weil reciprocity law cited above. 
This observation leads us to define an additive variant of the Somekawa $K$-group which 
allows an additive group $\Ga$ over $k$ as a coefficient. 
More precisely, let 
$\Wm$  be the big Witt group of length $m$ and 
$G_1,\ldots ,G_q$ split semi-abelian varieties over a perfect field $k$. 
Here, a split semi-abelian variety is a product  
of an abelian variety and a torus. 
We define the group 
\[
    K(k;\Wm, G_1,\ldots,G_q) := 
    \left(\bigoplus_{E/k:\,\mathrm{finite}} 
    \Wm(E) \otimes_{\Z} G_1(E) \otimes_{\Z} \cdots \otimes_{\Z} G_q(E)\right)/R
\] 
which we also call the \textit{Somekawa $K$-group}, 
where $R$ is the subgroup 
which produces ``the projection formula'' and 
``the Weil reciprocity law'' 
(= the residue theorem as in (\ref{eq:rt})) on $K(k;\Wm,G_1,\ldots,G_q)$ 
in a way similar to that in Kato and Somekawa's original $K$-groups, 
\Cf Definition \ref{def:Somekawa}.
%
However, the projection formula above implies that 
our Somekawa $K$-group becomes trivial when $k$ 
has positive characteristic and $q\ge1$ (Lem.\ \ref{lem:trivial}). 
Thus our main interest here is in the case of characteristic $0$. 
Let $\Ga$ denote the additive group and $\Gm$ the multiplicative group. 
For $\W_{1} =\Ga$ and $G_i = \Gm$ 
for $i\ge 1$, 
the Somekawa $K$-group has the following description:

\begin{theorem}[Thm.\ \ref{thm:main}]
    As $k$-vector spaces, we have 
    \[
        K(k;\Ga,\overbrace{\Gm,\ldots,\Gm}^{q-1}) \isomto \Omega_k^{q-1}.
    \]
\end{theorem}

\noindent
This isomorphism should be compared with Kato and Somekawa's theorem (\ref{eq:KS}).

For a projective smooth and geometrically connected curve $X$ 
over a perfect field $k$, 
the sequence  
\[
    \Omega_{k(X)}^1 \xrightarrow{\d} \bigoplus_{P \in X_0} \kP \xrightarrow{\Tr} k
\]
becomes a complex by the residue theorem, 
where $X_0$ is the set of closed points in $X$, 
$\d := \oplus_P \dP$ is the collection of the refined residue maps  
and  
$\Tr := \sum_P \Tr_{\kP/k}$ is the sum of the trace maps. 
Taking $\Ga$ and the Jacobian variety $J_X$ 
of the curve $X$, 
we relate the Somekawa $K$-group $K(k; \Ga, J_X)$ and the complex above 
as follows:

\begin{theorem}[Cor.\ \ref{cor:SK1+}]
    We assume that the curve $X$ has a $k$-rational point. 
    Then there is a canonical isomorphism  
    \[
      K(k; \Ga, J_X) \isomto \Ker(\Tr)/\Im(\d).
    \]
\end{theorem}

\sn
In the analogy we are pursuing, 
the homology group $\Ker(\Tr)/\Im(\d)$ of the above complex 
is an additive version of the group $V(X)$ 
which is used in the proof of the higher dimensional class field theory 
\cite{Bloch81}, \cite{SSaito85b}.

\sn
\textit{Outline of this paper.}\ 
In Section \ref{sec:Witt}, 
we recall some properties of 
the generalized Witt group $\WS(k)$ for a field $k$, 
which is a generalization of 
both 
the ordinary Witt group and the big Witt group 
by \cite{HM01}, Section 1.  
We also recall  
the generalized de Rham-Witt complex $\WSOmega_k^{\bullet}$, 
following \cite{HM01}. 
After recalling R\"ulling's trace map and the residue map 
of the generalized de Rham-Witt complex (\cite{Rue07}),  
we study some functorial properties of 
the refined residue map 
$\dP: \WSOmega_{k(C)}^q \to \WSOmega_{k(P)}^{q-1}$ 
under the assumption that the base field $k$ is perfect. 
In Section \ref{sec:Somekawa}, 
we recall the definition and some properties 
of Mackey functors. 
Then combining Somekawa's local symbol 
and the residue map, 
we introduce the Somekawa $K$-group 
of the form 
$K(k;\WS, G_1,\ldots ,G_q)$  
for split semi-abelian varieties $G_1,\ldots, G_q$ 
over a perfect field $k$ 
as a quotient of the Mackey product 
$\WS \otimesM G_1 \otimesM \cdots \otimesM G_q (k)$ (Def.\ 2.10). 
Finally, 
we show that this group has the following 
descriptions: 
\begin{align*}
&K(k;\WS) \isomto \WS(k)\quad \mathrm{(Lem.\ \ref{lem:isom})}, \\
&K(k;\WS, G_1,\ldots ,G_q) = 0,\ 
  \mathrm{if}\ \Char(k)>0,\ q>1\quad \mathrm{(Lem.\ \ref{lem:trivial})},\\
& K(k;\Ga, \Gm,\ldots ,\Gm) \isomto \Omega_k^{q-1}\quad 
\mathrm{(Thm.\ \ref{thm:main})}.
\end{align*}
In Section \ref{sec:mix}, 
we study the additive counterpart of Section 2 in \cite{Som90}.  
Following Akhtar \cite{Akhtar04a}, 
the Somekawa $K$-group will be extended to 
$K(k;\WS, G_1,\ldots ,G_q, \cCH_0(X))$ 
where $\cCH_0(X)$ 
is the Mackey functor given by 
the Chow group of $0$-cycles  
on a variety $X$ over $k$
(Def.\ \ref{def:mix}). 
We will show that, 
for a smooth projective and geometrically connected curve $X$, 
\begin{align*}
  K(k;\WS, \cCH_0(X)) &\isomto 
    \WS(k),\ \mbox{if\ }\Char(k)>0\quad \mathrm{(Lem.\ \ref{lem:mix}\ (ii)),\ and}\\
  K(k;\Ga, \cCH_0(X)) &\isomto 
    \Coker\left(\d:\Omega^1_{k(X)} \to \bigoplus_{P\in X_0}\kP\right)\quad 
\mathrm{(Thm.\ \ref{thm:SK1+}\ (ii)).}
\end{align*}

Throughout this note, 
we mean 
by an \textit{algebraic group} 
a smooth connected and commutative group scheme over a field. 
For a field $F$, we denote by $\Char(F)$ the characteristic of $F$. 

Recently  
F.~Ivorra and K.~R\"ulling  
defined a  
Milnor-type $K$-group 
$T(\M_1,\ldots ,\M_q)(\Spec k)$ 
for Nisnevich sheaves with transfers 
$\M_1,\ldots , \M_q$ 
on the category of regular schemes over a field $k$ 
with dimension $\le 1$ satisfying several axioms (\cite{IR}). 
They show also 
\[
    T(\Ga, \overbrace{\Gm,\ldots , \Gm}^{q-1})(\Spec k) \isomto \Omega_{k}^{q-1}
\] 
when $k$ is perfect. 

\medskip\noindent
\textit{Acknowledgments.} 
A part of this note was written 
during a stay of the author at the Duisburg-Essen university. 
He thanks the institute for its hospitality. 
The author is indebted to Henrik Russell 
for his encouragement during the author's stay in Essen 
and helpful comments on this manuscript. 
He thanks also Kay R\"ulling for kindly informing 
us the results with Florian Ivorra \cite{IR}. 
Finally, thanks are due to the referee for 
valuable advice which enabled 
simplifying the original paper 
and to Professor Anthony Bak 
for reading the manuscript and making a number of helpful suggestions. 
This work was supported by KAKENHI 30532551.

\section{Witt groups and local symbols}
\label{sec:Witt}
Let $A$ be a (commutative) ring. 
A subset $S \subset \mathbb{N}$ of the set of positive integers 
$\mathbb{N}:= \Z_{>0}$ 
is called a \textit{truncation set} 
if $S\neq \emptyset$ and if for any $n\in S$, 
all its divisors are in $S$. 
The \textit{generalized Witt ring} $\WSA$ is defined to be $A^S$ as a set 
whose ring structure is determined by the condition that 
the ghost map 
\begin{equation}
    \label{eq:ghost}  
    \gh = (\gh_s)_{s\in S}: \WSA \to A^S
\end{equation}
defined by 
$\gh_s((a_t)_{t\in S}) := \sum_{d\mid s} da_d^{s/d}$,  
is a natural transformation of functors of rings (\cite{HM01}, Sect.\ 1). 
We denote by $[-]:A \to \WSA$ 
the Teichm\"uller map defined by 
$[a]:= (a,0, 0, \ldots )$. 

\begin{lemma}
    $\mathrm{(i)}$ The Teichm\"uller map is multiplicative, 
    that is, $[ab] = [a][b]$ for any $a,b\in A$. 
    The unit and zero in $\WSA$ are $[1]$ and $[0]$ 
    respectively. 

    \sn
    $\mathrm{(ii)}$ Any ring homomorphism $\phi:A\to B$ 
    induces a ring homomorphism $\WSA \to \WSB$ 
    defined by $(a_s)_{s\in S} \mapsto (\phi(a_s))_{s\in S}$. 
    If $\phi$ is surjective (resp.\ injective), then the same holds 
    on Witt rings.  


    \sn 
    $\mathrm{(iii)}$ For two truncation sets $T\subset S$, 
    the restriction map 
    $\WS(A) \to \W_T(A)$ 
    defined by $(a_s)_{s\in S} \mapsto (a_{s})_{s\in T}$ 
    is surjective. 
\end{lemma}

\noindent
For each $n\in \mathbb{N}$, 
$S/n:= \{ d\in \mathbb{N}\ |\ nd\in S\}$ 
is a truncation set. 
The Frobenius and the Verschiebung 
\[
    F_n:\WSA \to \W_{S/n}(A),\quad V_n:\W_{S/n}(A) \to \WSA 
\]
are defined by the rules $\mathrm{gh}_s\circ F_n = \mathrm{gh}_{sn}$ and 
\[
    (V_n(a))_s := 
        \begin{cases}
            a_{s/n}, &\mathrm{if}\ n\mid s,\\ 
            0,& \mathrm{otherwise}, 
        \end{cases}               
\]
respectively.

\begin{lemma}
    \label{lem:Witt}
    $\mathrm{(i)}$ If $F_1 = \Id$, $F_m\circ F_n = F_{mn}$ 
    and $V_1 = \Id$, then $V_m\circ V_n = V_{mn}$, 
    where $\Id$ is the identity map.

    \sn
    $\mathrm{(ii)}$ For $w = (w_s)_{s\in S} \in \WSA$, 
    \[
      w = \sum_{s\in S} V_s([w_s]).
    \]

    \sn
    $\mathrm{(iii)}$ If $(m,n) = 1$, then $F_m \circ V_n = V_n\circ F_m$.

    \sn
    $\mathrm{(iv)}$ $F_n\circ V_n = n$.

    \sn
    $\mathrm{(v)}$ $F_m\circ V_n ([a]) = (m,n)V_{n/(m,n)}([a]^{m/(m,n)})$ 
    for all $a \in A$.

    \sn
    $\mathrm{(vi)}$ $V_n([a])V_r([b]) = (n,r)V_{nr/(n,r)}([a]^{r/(n,r)}[b]^{n/(n,r)})$ 
    for all $a,b \in A$. 

    \sn
    $\mathrm{(vii)}$ $[a]V_n(w) = V_n([a]^nw)$ 
    for all $a\in A$ and $w \in \WSA$.
\end{lemma}
We define 
$\WmA := \W_{\{1,2,\ldots,m\}}(A)$ which is called the \textit{big Witt ring} of length $m$. 
For a fixed prime $p$,  
\[
    W(A):= \W_{\{1,p,p^2,\ldots \}}(A)\quad \mbox{and}\quad W_m(A):= \W_{\{1,p,\ldots ,p^{m-1}\}}(A).
\] 

Let $p$ be an odd prime or $0$. 
For a $\Z_{(p)}$-algebra $A$ 
and a truncation set $S$, 
the \textit{de Rham-Witt complex} $\WSOmega_A^{\bullet}$ 
is a differential graded algebra over $A$ 
which generalizes 
the de Rham complex $\Omega_A^{\bullet} := \Omega_{A/\Z}^{\bullet}$ 
and comes 
equipped with the Verschiebung  
$V_n:\W_{S/n}\Omega_{A}^{\bullet} \to \WSOmega_A^{\bullet}$ 
and the Frobenius 
$F_n: \WSOmega_A^{\bullet} \to \W_{S/n}\Omega_{A}^{\bullet}$ 
for each $n > 0$. 
The construction of $\WSOmega_A^{\bullet}$ 
and some properties of it are given in \cite{HM04} and \cite{Ill79}.
If $S$ is finite, 
there is a surjective map of differential graded rings 
$\Omega_{\WS(A)}^{\bullet} \to \WSOmega_A^{\bullet}$ 
and the map becomes 
bijective when $S=\{1\}$ or in degree $0$, 
namely, 
$\W_{\{1\}}\Omega_A^{\bullet} \simeq \Omega_{A/\Z}^{\bullet}$ 
is the absolute de Rham complex and 
$\WS\Omega_A^0 \simeq \WS(A)$ is the Witt ring.  

\begin{proposition}[\cite{HM01}, Sect.\ 1.2, \cite{Rue07}, Prop.\ 1.19]
    \label{prop:dW}
    Set 
    $P := \{1,p,p^2,\ldots \}$ 
    and 
    for a differential graded algebra $(\Omega,d)$ 
    and a positive integer $j$ 
    we denote 
    $\Omega(j^{-1}) :=(\Omega, j^{-1}d)$. 
    Let $A$ be a $\Z_{(p)}$-algebra and $S$ a finite truncation set.  

    \sn
    $\mathrm{(i)}$  
    Then, there is a canonical isomorphism 
    \[
        \WSOmega_A^{\bullet} \isomto 
        \prod_{j \in \mathbb{N},\, (j,p) = 1}\W_{P\cap S/j}\Omega_A^{\bullet}(j^{-1})
    \]
    as differential graded algebras 
    given by $\omega \mapsto F_j(\omega)|_{P\cap S/j}.$

    \sn
    $\mathrm{(ii)}$ 
    If $p$ is nilpotent on $A$, then
    for an \'etale homomorphism $A\to B$ of $\Z_{(p)}$-algebras, 
    there is an isomorphism of $\WS(B)$-modules 
    \[
        \WS(B)\otimes_{\WS(A)} \WSOmega_A^q \simeq \WSOmega_B^q
    \] 
    for any $q$.
\end{proposition}

\begin{theorem}[\cite{Rue07}, Thm.\ 2.6]
    \label{thm:trace}
    Let $F$ be a field with $\Char(F) \neq 2$, 
    and $S$ a finite truncation set.
    For a finite field extension  $E/F$, 
    there is a map of differential graded 
    $\WSOmega_F^{\bullet}$-modules
    \[
      \Tr_{E/F}= \Tr^{\bullet}_{E/F}: \WSOmegaEb \to \WSOmegaFb
    \] 
    satisfying the following properties:

    \sn
    $\mathrm{(a)}$ $\Tr_{E/F}^0:\WSE\to \WSF$ is the trace map 
    on Witt rings and coincides with 
    the one on the de Rham complexes 
    $\Tr_{E/F}:\Omega_{E}^{\bullet} \to \Omega_F^{\bullet}$ {\rm (\cite{Kunz})} 
    if $S = \{1\}$.

    \sn
    $\mathrm{(b)}$ If $E/F$ is separable, then 
    we identify $\WSE\otimes_{\WSF} \WSOmegaFq \simeq \WSOmegaEq$ 
    and the map $\Tr_{E/F}^q$ is given by 
    $\Tr^0_{E/F} \otimes \Id$, where $\Id$ is the identity map.

    \sn
    $\mathrm{(c)}$ For finite field extensions $F \subset E_1 \subset E_2$, 
    we have $\Tr_{E_2/F} = \Tr_{E_1/F}\circ \Tr_{E_2/E_1}$.

    \sn
    $\mathrm{(d)}$ The trace map commutes with $V_n,F_n$ and restriction maps.
\end{theorem}

\begin{proposition}[\cite{Rue07}, Prop.\ 2.12]
    \label{prop:res}
    Let $A$ be a $\Z_{(p)}$-algebra 
    and $S$ a finite truncation set. 
    There is a map  
    $\Res = \Res_{t,S}^q : \WS\Omega_{\Att}^{q} \to \WS\Omega_{A}^{q-1}$
    satisfying the following properties:

    \sn
    $\mathrm{(a)}$ $\Res^q(\alpha\omega) = \alpha \Res_t^{q-i}(\omega)$ 
    for $\alpha \in \WS\Omega_{A}^i$ and $\omega \in \WSOmega_{\Att}^{q-i}$. 

    \sn
    $\mathrm{(b)}$ $\Res$ 
    is a natural transformation with respect to $A$.

    \sn
    $\mathrm{(c)}$ $\Res$ commutes with $d$, $V_n,F_n$ and restriction. 

    \sn
    $\mathrm{(d)}$ If $u\in (A[\![t]\!])^{\times}$ and $\tau = tu$, then 
    $\Res_t = \Res_{\tau}$.

    \sn
    $\mathrm{(e)}$ $\Res(\omega) = 0$ 
    for $\omega \in \WSOmega^q_{A[\![t]\!]}$ or $\omega \in \WSOmega_{A[1/t]}^q$. 

    \sn
    $\mathrm{(f)}$ If $\omega \in \WSOmega_{A[\![t]\!]}^{q-1}$, 
    then $\Res(\omega \dlog[t]) = \omega(0)$, 
    where 
    $[t]:= (t,0,\ldots ,0)$ is the Teichm\"uller lift of $t$, 
    $\dlog [t] = d[t]/[t]$ and 
    $\omega(0)$ is the image of $\omega$ by the natural map 
    $\WSOmega_{A[\![t]\!]}^{q-1} \to \WSOmega_{A}^{q-1}$. 

    \sn
    $\mathrm{(g)}$
    For $a,b\in A$ and $i,j \in \Z$, 
    \begin{align*}
        &\Res(V_n([at^j])dV_m([bt^i])) \\ 
        &\qquad = 
        \begin{cases}
            \mathrm{sgn}(i)(i,j)V_{mn/(m,n)}([a]^{m/(m,n)}[b]^{n/(m,n)}), 
            &\mathit{if}\ jm+in = 0\ \mathit{and}\ i\neq 0\\
            0,& \mathit{otherwise},
        \end{cases}\notag
    \end{align*} 
    where $\mathrm{sgn}(i) = i/|i|$.
\end{proposition}

Henceforth, and until the end of this section, 
we suppose that $k$ is a \textit{perfect} field 
with $\Char(k) = p$, 
where $p$ is an odd prime or $0$. 
For a projective smooth curve $C$ over $k$, 
let $K = k(C)$ be the function field of $C$. 
For each closed point $P$ of $C$, 
let $\KP$ be the fraction field of the completion 
$\Ohat_{C,P}$ of $\O_{C,P}$ by 
the normalized valuation $v_P$ 
and $\kP$ the residue field of $\KP$. 
The complete discrete valuation field $\KP$ 
is canonically isomorphic to $\kPtt$. 
This allows us to define the refined residue map 
\begin{equation}
    \label{def:res}
    \dP= \dP^q: 
    \WS\Omega_{K}^q \to \WS\Omega_{K_P}^{q} \stackrel{\Res_t^q}{\to} \WS\Omega_{\kP}^{q-1}.
\end{equation}
By the above proposition, 
$\dP$ is independent on the choice of $t$. 
The ordinary residue map in \cite{Rue07} (and also \cite{Kunz})  
is the composition $\Res_P:= \Tr_{\kP/k}\circ \dP$. 

\begin{theorem}[Residue theorem, \cite{Rue07}, Thm.\ 2.19]
    \label{thm:res}
    Let $S$ be a finite truncation set. 
    The residue map 
    $\dP : \WS\Omega_{K_P}^{q} \to \WS\Omega_{\kP}^{q-1}$
    satisfies 
    \[
      \sum_{P\in C} \Tr_{\kP/k}\circ \dP(\omega) = 0
    \] 
    for any $\omega \in \WSOmega_K^q$.
\end{theorem}
The residue map is calculated as follows: 
The de Rham-Witt group $\WS\Omega_{K_P}^q$ 
decomposes into the \textit{$p$-typical} de Rham-Witt groups 
$W_m\Omega_{K_P}^q := \W_{\{1,p, \ldots ,p^{m-1}\}}\Omega_{K_P}^q$ 
(Prop.\ \ref{prop:dW}). 
Any element $\omega \in W_m\Omega_{K_P}^q$ can be 
written uniquely as 
\begin{equation}
    \label{eq:exp}
      \omega = \sum_{j\in \Z}a_{0,j}[t]^j + b_{0,j}[t]^{j-1}d[t] 
    + \sum_{j \in \Z\ssm p\Z,\, 1\le s<m} V^s(a_{s,j}[t]^j) + d V^s(b_{s,j}[t]^j).
\end{equation}
for some 
$a_{i,j}\in W_{m-i}\Omega_{\kP}^{q}$, 
$b_{i,j} \in W_{m-i}\Omega_{\kP}^{q-1}$ 
($a_{i,j} = b_{i,j} = 0$ for $j<<0$), 
where 
$V^s := V_{p^s}:W_{m-s}\Omega_{K_P}^q \to W_m\Omega_{K_P}^q$ 
is the Verschiebung 
(\cite{Rue07}, Lem.\ 2.9, see also Rem.\ 2.10). 
The residue map is now given by $\dP(\omega) := b_{0,0}$.  

\begin{definition}
    \label{def:ls} 
    The natural inclusion $K \inj \KP$ 
    gives the \textit{local symbol} 
    \[
        \dP:\WS(K) \otimes_{\Z}K^{\times} \to \WS(\kP);\ 
        w \otimes f \mapsto \dP(w \dlog [f]) 
    \] 
    where $\dlog[f] := d[f]/[f]$. 
    The image of $w\otimes f \in \WS(K) \otimes_{\Z}K^{\times}$ 
    by the local symbol map is denoted by $\dP(w,f) := \dP(w \dlog [f])$. 
    This is essentially Witt's residue symbol (\Cf \cite{Rue07}, Rem.\ 2.13). 
\end{definition}
The local symbol $\dP$ satisfies the conditions of Serre's local symbol 
(\cite{AGCF}, Chap.\ 3, Sect.\ 1). 
In particular, we have 
\begin{equation}
    \label{eq:residue}
    \dP(w, f) = v_P(f)w(P)
\end{equation}
if $w$ is in $\WS(\Ohat_{C,P})$, 
where $w(P)$ is the image of $w$ by the canonical map $\WS(\Ohat_{C,P}) \to \WS(\kP)$. 
The boundary map $\dP:K_q^M(K) \to K_{q-1}^M(\kP)$ on 
Milnor $K$-groups (\cite{BT73}, Prop.\ 4.3) 
and 
the residue map $\dP:\WSOmega_K^q \to \WSOmega_{\KP}^{q-1}$ 
are compatible in the following sense: 
For some field $F$, we denote the image of 
$x_1 \otimes \cdots \otimes x_q 
\in F^{\times}\otimes_{\Z}\cdots \otimes_{\Z}F^{\times}$ in 
$K_q^M(F)$ by $\{x_1,\ldots ,x_q\}$ as usual. 
For $q>0$, define 
\[
  \dlog:K_q^M(F) \to \WSOmega_F^q
\] 
by 
$\{x_1,\ldots, x_q\} \mapsto \dlog [x_1] \cdots \dlog [x_q]$.  
When $q=0$, $\dlog:\Z \to \WS(F)$ is the canonical map. 

\begin{lemma}
    \label{lem:tameres} 
    Let $C$ be a projective smooth curve over $k$ 
    and $P$ a closed point of $C$. 
    Then the following diagram is commutative:
    \[
      \xymatrix@C=15mm{
        K_q^M(k(C)) \ar[r]^{\dlog}\ar[d]_{\dP}& \WSOmega_{k(C)}^q \ar[d]^{\dP}\\
        K_{q-1}^M(\kP) \ar[r]^{\dlog} & \WSOmega_{\kP}^{q-1}.
      }
    \]
\end{lemma}
\begin{proof}
    The valuation map $v_P$ is the boundary map $\dP:K_1(K)\to K_0(\kP)$. 
    From the equality (\ref{eq:residue}), we may assume $q>1$. 
    By considering the completion $k(C)_P$ at $P$, 
    it is enough to show that 
    the following diagram is commutative
    \[
      \xymatrix@C=15mm{
        K_q^M(K) \ar[r]^{\dlog}\ar[d]_{\d}& \WSOmega_{K}^q \ar[d]^{\Res_t^q}\\
        K_{q-1}^M(k) \ar[r]^{\dlog} & \WSOmega_{k}^{q-1}.
      }
    \]
    for a local field $K = k(\!(t)\!)$. 
    It is known that the Milnor $K$-group 
    $K_q^M(K)$ is generated by symbols of the form 
    $\{u_1,\ldots, u_{q-1}, f\}$ with 
    $v_K(u_i) =0$ for all $i$ and $f\in \Kt$. 
    For an element $\{u_1,\ldots, u_{q-1}, f\}$ of this form, 
    we have 
    $\d(\{u_1,\ldots ,u_{q-1}, f\}) = v_K(f) \{\ol{u}_1,\ldots ,\ol{u}_{q-1}\}$  
    in $K_{q-1}^M(k)$, where 
    $\ol{u}_i$ is the image of $u_i$ in $k^{\times}$ (\cite{BT73}, Prop.\ 4.5). 
    On the other hand, 
    write 
    $f = ut^{v_K(f)}$ with $u\in \O_K^{\times}$. 
    \begin{align*}
      \Res_t \circ \dlog (\{u_1,\ldots ,u_{q-1},f\}) 
      & = \Res_t(\dlog [u_1] \cdots \dlog [u_{q-1}]\dlog[u]) \\ 
      & \quad + v_K(f)\Res_t(\dlog [u_1] \cdots \dlog [u_{q-1}]\dlog[t])\\
      & \stackrel{(*)}{=} v_K(f)\dlog [\ol{u}_1]\cdots \dlog [\ol{u}_{q-1}]. 
    \end{align*}
    Here,  
    the last equality $(*)$ follows from Proposition \ref{prop:res} (e) and (f).
\end{proof}

\begin{lemma}
    \label{lem:rescan}
    Let $C' \to C$ be a dominant morphism 
    of smooth projective curves over $k$. 
    For any closed point $P'$ in $C'$, 
    we denote by the point $P$ in $C$ the image of $P'$. 
    Then, we have the following commutative diagram
    \[
    \xymatrix@C=15mm{
      \WS\Omega_{k(C')}^q \ar[r]^{\d_{P'}} & \WS\Omega_{k(P')}^{q-1} \\
      \WS\Omega_{k(C)}^{q}\ar[r]^{\dP}\ar[u] & \WS\Omega_{\kP}^{q-1}\ar[u]_{e(P'/P)},
    }
    \]
    for $q\ge 1$, 
    where the left vertical map is the natural map and the right 
    one is the map multiplication by the ramification index $e(P'/P)$ 
    at $P'$ over $P$.
\end{lemma}
\begin{proof}
    By considering the completions $k(C)_P$ and $k(C')_{P'}$, 
    it is enough to show, 
    for a finite extension $K'/K$ of local fields 
    with residue field extension $k'/k$ and $\Char(K)>0$,  
    \[
    \xymatrix@C=15mm{
      \WS\Omega_{K'}^q \ar[r]^{\d_{K'}} & \WS\Omega_{k'}^{q-1} \\
      \WS\Omega_{K}^{q}\ar[r]^{\d_K}\ar[u] & \WS\Omega_k^{q-1}\ar[u]_{e},
    }
    \]
    where $\d_K$ and $\d_{K'}$ are residue maps 
    and $e$ is the multiplication of the ramification index of $K'/K$. 
    The de Rham-Witt group $\WS\Omega_{K}^q$ 
    decomposes into the $p$-typical de Rham-Witt groups 
    $W_m\Omega_{K}^q$ (Prop.\ \ref{prop:dW}, (i)) 
    and considering the restriction maps 
    it is enough to show the equality 
    \begin{equation}
      \label{eq:dK}
     \d_{K'} = e \d_K. 
    \end{equation}
    First we assume that $K'/K$ is unramified.  
    If we fix a uniformizer  $t$ of $K$, 
    then $K' \simeq k'(\!(t)\!)$ and $K \simeq k(\!(t)\!)$.
    Since the residue map is natural (Prop.\ \ref{prop:res}), 
    we have 
    $\d_{K'}(\omega) = \d_K(\omega) \in W_m\Omega_{k'}^{q-1}$ for any $\omega\in W_m\Omega_{K}^q$. 
    Next we assume that the extension $K'/K$ is totally ramified 
    of degree $e$ 
    and thus $k = k'$. 
    Any $\omega\in W_m\Omega_{K}^q$
    can be written as in (\ref{eq:exp}), that is, 
    \[
      \omega 
    = \sum_{j\in \Z}a_{0,j}[t]^j + b_{0,j}[t]^{j-1}d[t] 
      + \sum_{\begin{smallmatrix}j \in \Z\ssm p\Z,\\ 1\le s<m\end{smallmatrix}} 
      V^s(a_{s,j}[t]^j) + d V^s(b_{s,j}[t]^j).
    \] 
    By Proposition \ref{prop:res}, we have 
    \begin{align*}
    \d_{K'}(\omega) = \sum_j b_{0,j}\d_{K'}([t]^{j-1}d[t]),\\
    \d_{K}(\omega) = \sum_j b_{0,j}\d_K([t]^{j-1}d[t]).
    \end{align*}
    Thus, we may assume $q=1$. 
    By Proposition \ref{prop:res} 
    and the injection $W_m(k) \inj W_m(\kbar)$, 
    where $\kbar$ is an algebraically closed field, 
    we can further assume that $k$ is algebraically closed. 
    If $\Char(k) = 0$, 
    then 
    for a uniformizer $t'$ of $K'$ 
    we have $t = (t')^eu$ with $u\in k[\![t']\!]^{\times}$. 
    There exists $v\in k[\![t']\!]^{\times}$ such that 
    $v^e = u$. 
    Replacing $t'$ by $t'v$ we have $t = (t')^e$. 
    From (\ref{eq:exp}), any element 
    $\omega \in W_m\Omega_K^1$ is written of the form
    \begin{align*}
    \omega 
    &= \sum_{j\in \Z}a_{0,j}[t]^j + b_{0,j}[t]^{j-1}d[t] 
      + \sum_{\begin{smallmatrix}j \in \Z\ssm p\Z,\\ 1\le s<m\end{smallmatrix}} 
      V^s(a_{s,j}[t]^j) + d V^s(b_{s,j}[t]^j)\\
    &= \sum_{j\in \Z}a_{0,j}[t']^{ej} + b_{0,j}[t']^{e(j-1)}d[(t')^e] 
      + \sum_{\begin{smallmatrix}j \in \Z\ssm p\Z,\\ 1\le s<m\end{smallmatrix}} 
      V^s(a_{s,j}[t']^{ej}) + d V^s(b_{s,j}[t']^{ej})\\
    &= \sum_{j\in \Z}a_{0,j}[t']^{ej} + e b_{0,j}[t']^{ej-1}d[t'] 
      + \sum_{\begin{smallmatrix}j \in \Z\ssm p\Z,\\ 1\le s<m\end{smallmatrix}} 
      V^s(a_{s,j}[t']^{ej}) +  d V^s(b_{s,j}[t']^{ej}). 
    \end{align*}
    Therefore $\d_{K'}(\omega) = eb_{0,0} = e\d_K(\omega)$ (Prop.\ \ref{prop:res} (g)). 
    Now we consider the case $\Char(k)=p>0$. 
    For a uniformizer $t'$ of $K'$ 
    we have 
    $t = (t')^eu$ with $u = u_0 + u_1 t' + \cdots \in k[\![t']\!]^{\times}$ ($u_0\neq 0$). 
    Put $\wt{u} := [u_0] + [u_1] t' + \cdots \in W(k)[\![t']\!]^{\times}$ 
    and is a lift under the natural map $W(k)[\![t']\!] \to k[\![t']\!]$. 
    There is an inclusion of $W(k)$-algebras $W(k)[\![t]\!] \inj W(k)[\![t']\!]$ 
    which sends $t$ to $(t')^e\wt{u}$.  
    It extends to $W(k)(\!(t)\!) \inj W(k)(\!(t')\!)$ and is also a lift of 
    the inclusion $K = k(\!(t)\!) \inj K' = k(\!(t')\!)$. 
    By the naturality given in Proposition \ref{prop:res} 
    and the natural map $W_m\Omega_{W(k)(\!(t)\!)}^1 \surj W_m\Omega_{k(\!(t)\!)}^1$ 
    is surjective, 
    it suffices to prove the equality (\ref{eq:dK}) 
    for $K,K',k$ replaced by 
    $W(k)(\!(t)\!), W(k)(\!(t')\!), W(k)$. 
    Now we denote by $E$ the fraction field of $W(k)$. 
    The inclusion $W(k)(\!(t)\!) \inj W(k)(\!(t')\!)$ 
    extends to $E(\!(t)\!) \inj E(\!(t')\!)$. 
    Since the natural map $W_m(W(k)) \inj W_m(E)$ 
    is injective, 
    the naturality of the residue map 
    reduces us to the case $\Char(k) = 0$. 
    Finally, 
    for arbitrary finite extension $K'/K$, 
    let $K^{\ur}$ be the maximal unramified subextension of $K'$ over $K$. 
    Then the required equality (\ref{eq:dK}) follows from 
    \[
      \d_{K'}(\omega) = e \d_{K^{\ur}}(\omega) = e \d_{K}(\omega)
    \]
    for any $\omega \in W_m\Omega_K^1$. 
\end{proof}

\section{Somekawa $K$-groups}
\label{sec:Somekawa}
Throughout this section, $k$ is a perfect field. 

\begin{definition}
    A \textit{Mackey functor} $A$ (over $k$)  
    is a contravariant 
    functor from the category of \'etale schemes over $k$ 
    to the category of abelian groups 
    equipped with a covariant structure 
    for finite morphisms 
    such that 
    $A(X_1 \sqcup X_2)  = A(X_1) \oplus A(X_2)$ 
    and if 
    \[
    \xymatrix@C=15mm{
      X' \ar[d]_{f'}\ar[r]^{g'} & X \ar[d]^{f} \\
      Y' \ar[r]^{g} & Y
    }
    \]
    is a Cartesian diagram, then the induced diagram 
    \[
    \xymatrix@C=15mm{
      A(X') \ar[r]^{{g'}_{\ast}} & A(X)\\
      A(Y') \ar[u]^{{f'}^{\ast}} \ar[r]^{{g}_{\ast}} & A(Y)\ar[u]_{f^{\ast}}
    }
    \]
    commutes. 
\end{definition}

For a Mackey functor $A$, 
we denote by $A(E)$ 
its value $A(\Spec(E))$  
for a field extension $E$ over $k$.

\begin{definition}
    For Mackey functors $A_1,\ldots ,A_q$, 
    their \textit{Mackey product} 
    $A_1\otimesM \cdots \otimesM A_q$ 
    is defined as follows: 
    For any finite field extension $k'/k$,    
    \[
        k' \mapsto 
        A_1\otimesM \cdots \otimesM A_q(k') 
         := \left(\bigoplus_{E/k':\, \mathrm{finite}} A_1(E) \otimes_{\Z}\cdots \otimes_{\Z}A_q(E)\right) /R, 
    \]
    where $R$ is the subgroup generated 
    by elements of the following form: 

    \sn
    (PF) 
    For any finite field extensions 
    $k' \subset E_1 \subset E_2$,  
    if $x_{i_0} \in A_{i_0}(E_2)$ and $x_i \in A_i(E_1)$ 
    for all $i\neq i_0$, then 
    \[
        j^{\ast}(x_1) \otimes \cdots \otimes x_{i_0} \otimes \cdots \otimes j^{\ast}(x_q) 
        - x_1\otimes \cdots \otimes j_{\ast}(x_{i_0})\otimes \cdots \otimes x_q,
    \]
    where $j:\Spec(E_2) \to \Spec(E_1)$ is the canonical map. 
\end{definition}

\sn
This product gives a tensor product in the abelian category 
of Mackey functors with unit $\Z:k' \mapsto \Z$. 
We write $\{x_1,\ldots,x_q\}_{E/k}$ 
for the image of 
$x_1 \otimes \cdots \otimes x_q \in 
A_1(E) \otimes_{\Z}\cdots \otimes_{\Z}A_q(E)$ in the product 
$A_1\otimesM \cdots \otimesM A_q(k)$. 
For any finite field extension $k'/k$ 
and the canonical map $j=j_{k'/k}:k\inj k'$, 
the pull-back 
\[
    \Res_{k'/k} := j^{\ast}: A_1\otimesM \cdots \otimesM A_q(k) \longrightarrow 
    A_1\otimesM \cdots \otimesM A_q(k')
\] 
is called the \textit{restriction map}. 
We recall here the construction of the restriction 
$\Res_{k'/k}$. 
For any finite extension $E/k$, 
we have $E\otimes_k k' = \bigoplus_{i=1}^n E_i$ 
where $E_i$ is a separable field extension over $k'$. 
Denoting $j_i:=j_{E_{i}/E}:E \inj E_{i}$, the restriction map is 
\[
  \Res_{k'/k}(\{x_1,\ldots ,x_q\}_{E/k}) := 
    \sum_{i=1}^n \{j_i^{\ast}(x_1),\ldots ,j_i^{\ast}(x_q)\}_{E_i/k'}. 
\]
On the other hand, the push-forward 
\[
  N_{k'/k} := j_{\ast}: A_1\otimesM \cdots \otimesM A_q(k') \longrightarrow 
    A_1\otimesM \cdots \otimesM A_q(k)
\] 
is given by $N_{k'/k}(\{x_1,\ldots ,x_q\}_{E'/k'}) = \{x_1,\ldots, x_q\}_{E'/k}$ 
and is called the \textit{norm map}. 

An algebraic group $G$ over $k$ 
forms a Mackey functor which is given by $E\mapsto G(E)$.  
For a field extension $E_2/E_1$, 
the pull-back 
$j^{\ast}= \Res_{E_2/E_1}:G(E_1)\inj G(E_2)$ 
is the canonical map given by $j:E_1 \inj E_2$. 
For simplicity, 
we sometimes identify elements in $G(E_1)$ 
with their images in $G(E_2)$ 
and 
the restriction map 
$\Res_{E_2/E_1}$ will be omitted. 
If the extension $E_2/E_1$ is finite, 
the push-forward is written as 
$j_{\ast} = N_{E_2/E_1}:G(E_2)\to G(E_1)$ 
and is referred to as the norm map.
For the function field $K = k(C)$ 
of a proper smooth curve $C$ over $k$ and 
a semi-abelian variety $G$ over $k$, 
the local symbol map 
\begin{equation}
    \label{eq:local_symbol sa}
    \dP:G(\KP) \otimes_{\Z}\KPt \to G(\kP)
\end{equation}
is defined in $\cite{Som90}$.  
Denoting by $\dP(g,f) := \dP(g\otimes f) \in G(\kP)$ 
it satisfies $\dP(g, f) = v_P(f)g(P)$ 
if $g \in G(\Ohat_{C,P})$,
where $g(P)\in G(\kP)$ denotes the image of $g$ under the canonical map 
$G(\Ohat_{C,P}) \to G(\kP)$, 
$v_P$ is the valuation associated to the point $P$.  
As a special case, for the multiplicative group $G = \Gm$, 
the local symbol map is the tame symbol map 
(= the boundary map $\d_P: K_2(K) \to K_1(\kP)$) 
which is given by
\begin{equation}
    \label{eq:tame}
    \dP:\Gm(K)\otimes_{\Z}K^{\times} \to \Gm(\kP);\ 
    g\otimes f \mapsto (-1)^{v_P(g)v_P(f)}\frac{g^{v_P(f)}}{f^{v_P(g)}}(P).
\end{equation}


\begin{definition}
    \label{def:mult}
    Let $K$ be a discrete valuation field 
    with residue field $k$ 
    and 
    $G = T\times A$ a split semi-abelian variety 
    with torus $T$ and an abelian variety $A$ over $k$. 
    For any $g = (t,x) \in T(K) \times A(K) = G(K)$, 
    let $L/K$ be a finite unramified extension 
    with residue field $E$ 
    such that $T_E \simeq ({\Gm})^{\oplus n}$. 
    We denote by $(g_i)_{1\le i\le n}\in \Gm(L)^{\oplus n}$ 
    the image of $t$ by 
    $T(K) \inj T(L) \simeq \Gm(L)^{\oplus n}$. 
    We define  
    \[
    m_K(g) := \begin{cases}
                  \infty,& \mbox{if $t = 0$},\\
                  \min\{v_L(1-g_i)\ |\  1\le i \le n\}, & \mbox{otherwise}.
                 \end{cases}
    \] 
    Here 
    $v_L$ is the valuation of $L$.
    \end{definition}

    \begin{definition}
    \label{def:Somekawa}
    Let $S$ be a finite truncation set and  
    $G_1,\ldots ,G_{q}$ split semi-abelian varieties over $k$. 
    The Milnor type $K$-group 
    $K(k;\WS, G_1,\ldots,G_{q})$ 
    which we also call the \textit{Somekawa $K$-group} 
    is given by the quotient 
    \[ 
      K(k;\WS, G_1,\ldots,G_{q}) 
     := \left(\WS \otimesM G_1\otimesM \cdots \otimesM G_{q}(k)\right)/R
    \]
    modulo the subgroup $R$ generated by elements of the following form:

    \sn
    (WR) Let $K = k(C)$ be the function field of 
    a projective smooth curve $C$ over $k$. 
    Put $G_0:= \WS$.  
    For $g_i\in G_i(K)$ and $f\in K^{\times}$,  
    assume that for each closed point $P$ in $C$ there exists 
    $i(P)$ ($0\le i(P) \le q$) 
    such that $g_i \in G_i(\Ohat_{C,P})$ for all $i\neq i(P)$. 
    If $g_0 \not \in \WS(\Ohat_{C,P})$ and $q\ge 1$, 
    then we further assume, for any $s\in S$ with $v_P(\gh_s(g_0)) <0$,   
    \begin{equation}
    \label{eq:modulus}
       (m+1)v_P(\gh_s(g_0)) + \sum_{i=1}^{q}m_P(g_i) + v_P(1-f) \ge 0,
    \end{equation}
    where $m_P(g_i) := m_{K_P}(g_i)$ (Def.\ \ref{def:mult}), 
    $\gh_s:\WS(K)\to K$ is the ghost map (\ref{eq:ghost}) 
    and $m := \max\{s \in S\}$. 
    Then 
    \[ 
      \sum_{P \in C_0}g_0(P)\otimes \cdots \otimes \dP(g_{i(P)}, f)\otimes \cdots \otimes g_{q}(P) \in R.
    \]
    Here $C_0$ is the set of closed points in $C$, 
    $g_i(P) \in G_i(k(P))$ 
    denotes the image of $g_i$ under the canonical map $G_i(\Ohat_{C,P}) \to G_i(\kP)$  
     and 
    $\dP:G_i(K_P)\otimes_{\Z} K^{\times}_P \to G_i(\kP)$ 
    is the local symbol (Def.\ \ref{def:ls} for $i=0$ and (\ref{eq:local_symbol sa}) for $i>0$). 
\end{definition}

Note that the expression above 
$g_0(P)\otimes \cdots \otimes \dP(g_{i(P)},f)\otimes \cdots \otimes g_{q}(P)$ 
is independent of the choice of $i(P)$ in the case 
$g_i \in G_i(\Ohat_{C,P})$ for all $i$. 
In fact, 
we have  
\[
g_0(P)\otimes \cdots \otimes \dP(g_{i(P)}, f)\otimes \cdots \otimes g_{q}(P)
  = v_P(f)g_0(P) \otimes \cdots \otimes g_{q-1}(P).
\]
We also write $\s{x,x_1,\ldots,x_q}_{E/k}$ for the image of 
$x\otimes x_1\otimes \cdots \otimes x_q 
\in \WS(E) \otimes_{\Z}G_1(E)\otimes_{\Z}\cdots \otimes_{\Z}G_q(E)$ in $K(k;\WS, G_1,\ldots ,G_q)$. 
By multiplication on the first argument, 
$K(k, \WS, G_1,\ldots ,G_q)$ has the structure of 
a $\WS(k)$-module. 
The various functorial properties of the Somekawa $K$-groups 
hold as stated in \cite{Som90}:  
For a finite field extension $k'/k$, 
set $K(k';\WS,G_1,\ldots,G_q) := K(k';\WS\otimes_k k', G_1\otimes_k k' ,\ldots ,G_q\otimes_kk')$. 
The restriction map on the Mackey product induces a canonical homomorphism 
\[
  \Res_{k'/k}:K(k;\WS, G_1,\ldots,G_q) \longrightarrow K(k';\WS, G_1,\ldots,G_q). 
\]
The norm map 
\[
  N_{k'/k}:K(k';\WS,G_1,\ldots,G_q) \longrightarrow K(k;\WS,G_1,\ldots,G_q)
\] 
is defined on symbols by $N_{k'/k}(\s{x,x_1,\ldots ,x_q}_{E'/k'}) = \s{x,x_1,\ldots ,x_q}_{E'/k}$. 

%
\begin{lemma}
    \label{lem:isom}
    $\mathrm{(i)}$ There exists an isomorphism 
    \[
      K(k;\WS) \isomto \WS(k). 
    \]

    \sn
    $\mathrm{(ii)}$ For split semi-abelian varieties $G$ and $G'$ over $k$, we have 
    \[
      K(k;\WS, G\oplus G') \simeq K(k; \WS,G) \oplus K(k;\WS,G').
    \]
    \end{lemma}
\begin{proof}
    (i) 
    We define a map $\phi:\WS(k) \to K(k;\WS)$ by 
    $\phi(w) := \s{w}_{k/k}$. 
    This is surjective from the property (PF). 
    For the algebraic closure $\kbar$ of $k$, 
    considering the commutative diagram 
    \[
    \xymatrix@C=15mm{
    \WS(k) \ar@{^{(}->}[d]\ar@{->>}[r]^{\phi} & K(k;\WS) \ar[d] \\
    \WS(\kbar) \ar@{->>}[r]^{\phi} & K(\kbar; \WS)
    }
    \]
    we may assume $k = \kbar$. 
    The Somekawa $K$-group 
    $K(k;\WS)$ is now the quotient of $\WS(k)$ by 
    the subgroup generated by the relations of the form (WR). 
    By the residue theorem (Thm.\ \ref{thm:res}), 
    the local symbol map satisfies 
    $\Sigma_P\d_P(w, f) = 0$ in $\WS(k)$ 
    for some function field $K = k(C)$ of one variable over $k$, 
    $w\in \WS(K)$ and $f\in \Gm(K)$. 
    The assertion follows from it. 

    \sn
    (ii) Since the Mackey product $\ \otimesM\ $ commutes with the direct sum, 
    it is enough to show that 
    the maps 
    \begin{align*}
      &K(k;\WS,G) \to K(k;\WS, G\oplus G'); \{w, g\}_{E/k} \mapsto \{w, (g,1)\}_{E/k}, \mbox{and}\\
      &K(k;\WS, G\oplus G') \to K(k;\WS, G)\oplus K(k;\WS, G'); 
        \{w, (g,g')\}_{E/k} \mapsto (\{w, g\}_{E/k}, \{w,g'\}_{E/k})
    \end{align*}
    are well-defined. 
    They follow from $m_P((g,g')) = \min\{m_P(g), m_P(g')\}$ 
    for $(g,g')\in (G\oplus G')(K)$.
\end{proof}

For a finite field $k$ and $q\ge 2$, 
it is known that 
$G_1\otimesM \cdots \otimesM G_q(k) = 0$ 
for semi-abelian varieties $G_i$ 
and hence $K(k;G_1,\ldots ,G_q) = 0$ (\cite{Kahn92}).  
In the following, 
we show that 
the Mackey product $\WS\otimesM G_1\otimesM \cdots \otimesM G_q(k)$ 
becomes trivial 
for arbitrary perfect field of positive characteristic.  
However, 
the Mackey product $\Ga\otimesM \Ga (k)$  
is not trivial even if $k$ is a finite field. 
In fact, we always have a surjective map 
$\Ga \otimesM \Ga(k) \surj k$
given by $\{x,y\}_{E/k} \mapsto \Tr_{E/k}(xy)$ (see also \cite{H12}). 

\begin{lemma}
    \label{lem:trivial}
    Let $k$ be a perfect field with $\Char(k) = p>0$. 
    Let $G_1,\ldots, G_q$ be semi-abelian varieties over $k$ 
    for $q\ge 1$, 
    and $S$ a finite truncation set. 
    Then $\WS\otimesM G_1\otimesM \cdots \otimesM G_q = 0$ 
    as a Mackey functor. 
    In particular, we have   
    \[
      K(k;\WS,G_1,\ldots ,G_q)= 0 
    \] 
    for split semi-abelian varieties $G_1,\ldots ,G_q$.
\end{lemma}
\begin{proof}
    It is enough to show $\WS\otimesM G(k) = 0$ 
    for a semi-abelian variety $G$ over a perfect field $k$. 
    Since the generalized Witt group is  
    a finite product of Witt groups $W_m$ (Prop.\ \ref{prop:dW}),   
    the assertion is reduced to showing $W_m\otimesM G(k) =0$. 
    Any symbol $\{w,g\}_{E/k}$ in $W_m\otimesM G(k)$ 
    is in the image of the norm map 
    $N_{E/k}:W_m\otimesM G(E) \to W_m\otimesM G(k)$, 
    without loss of generality, 
    we may assume $E = k$ and show $\{w,g\}_{k/k} = 0$. 
    Take a finite field extension $j:k \inj E$ 
    such that $j^{\ast}(g) = p^m\wt{g}$ for some $\wt{g} \in G(E)$. 
    Since the trace maps on the Witt groups are surjective,  
    there exists $\wt{w} \in W_m(E)$ such that $\Tr_{E/k}(\wt{w}) =w$. 
    The projection formula (PF) implies  
    \begin{align*}
      \{w,g\}_{k/k} &= \{\Tr_{E/k}(\wt{w}), g\}_{k/k}\\
                    &= \{\wt{w}, j^{\ast}(g) \}_{E/k}\\
                    &= \{\wt{w}, p^m\wt{g} \}_{E/k}\\
                    &= \{p^m\wt{w}, \wt{g}   \}_{E/k}\\
                    &= 0.
    \end{align*}
    The assertion follows. 
\end{proof}

Recall that the isomorphism 
\[
    \psi: K(k;\overbrace{\Gm,\ldots ,\Gm}^{q}) \isomto K^M_{q}(k)
\] 
is given by $\psi(\{x_1,\ldots ,x_{q}\}_{E/k})= N_{E/k}(\{x_1,\ldots ,x_{q}\})$, 
where $N_{E/k}$ 
is the norm map on Milnor $K$-groups (\cite{Som90}, Thm.\ 1.4).
The inverse map $\phi$ is  
$\phi(\{x_1, \ldots ,x_q\}) = \{x_1 ,\ldots ,x_q\}_{k/k}$. 
In the same manner 
here we show that 
the Somekawa $K$-group 
taking the additive group $\Ga = \W_{1}$ and the multiplicative 
groups $\Gm$ coincides with the space of K\"ahler differentials.

\begin{theorem}
    \label{thm:main}
    Let $k$ be a perfect field. 
    As $k$-vector spaces, we have a canonical isomorphism 
    \[
    \psi:K(k;\Ga,\overbrace{\Gm,\ldots,\Gm}^{q-1}) \isomto \Omega_k^{q-1}.
    \] 
\end{theorem}

As the Milnor $K$-group $K_2^{M}(F)$ of a field $F$
is defined by the quotient of  
$F^{\times} \otimes_{\Z}F^{\times}$ 
modulo the Steinberg relation $x \otimes (1-x)$ 
for $x\in F^{\times}$, 
the space of the absolute K\"ahler differential 
$\Omega_F^1 :=  \Omega_{F/\Z}^1 = \W_{\{1\}}\Omega_F^1$   
has the so called \textit{Cathelineau relation} 
$x\dlog x + (1-x)\dlog (1-x) = 0$. 
Precisely, a surjective homomorphism of $F$-vector spaces
\begin{equation}
\label{eq:str}
  F\otimes_{\Z}\overbrace{F^{\times}\otimes_{\Z}\cdots \otimes_{\Z}F^{\times}}^{q-1} \to \Omega_F^{q-1}
\end{equation}
is defined by 
$x\otimes x_1\otimes \cdots \otimes x_{q-1} \mapsto x\dlog x_1 \cdots  \dlog x_{q-1} 
  = x\dlog x_1 \cdots \dlog x_{q-1}$ 
where $\dlog x := dx/x$ 
and its kernel is generated by elements of the forms (\cite{BE03b}, Lem.\ 4.1, \cite{BK86}, Lem.\ 4.2):
\begin{gather*}
 x\otimes x_1 \otimes \cdots \otimes x_{q-1}, \quad 
    \mathrm{with}\ x_i = x_j\ \mathrm{for\ some}\ i< j,\quad \mathrm{or}\\
 x\otimes x\otimes x_2 \otimes \cdots \otimes x_{q-1} 
    + (1-x)\otimes (1-x) \otimes x_2\otimes \cdots \otimes x_{q-1}.
\end{gather*}

\begin{proof}[Proof of Thm.\ \ref{thm:main}]
    From Lemma \ref{lem:isom}, we may assume $q>1$. 
    Since we are assuming the field $k$ is perfect, 
    the both sides are trivial if $\Char(k)>0$ (Lem.\ \ref{lem:trivial}).  
    Hence we assume $\Char(k)=0$.   

    \sn
    \textit{Definition of $\phi$}:
    Define a map 
    \[\phi:\Omega_k^{q-1} \to K(k;\Ga,\Gm,\ldots ,\Gm)\]  
    by $\phi(x\dlog x_1 \cdots \dlog x_{q-1}) := \s{x,x_1,\ldots ,x_{q-1}}_{k/k}$. 
    To show that the map $\phi$ is well-defined, 
    by (\ref{eq:str}) we need to show
    \begin{equation}
    \label{eq:relation}
    \begin{split}
    \{x,x_1,\ldots ,x_{q-1}\}_{k/k} = 0\quad 
      \mathrm{if}\ x_i = x_j \ \mathrm{for\ some}\ i<j,\quad \mathrm{and}\\
    \{x, x, x_2,\ldots,  x_{q-1}\}_{k/k} + 
    \{(1-x), (1-x), x_2,\ldots, x_{q-1}\}_{k/k} = 0.
    \end{split}
    \end{equation}
    Recall that the structure of a $k$-vector space 
    on the Somekawa $K$-group is defined by 
    $a\{x,x_1,\ldots ,x_{q-1}\}_{E/k} = \{ax, x_1,\ldots ,x_{q-1}\}_{E/k}$ 
    for $a\in k$. 
    For the first equality in (\ref{eq:relation}), 
    we have to show $\{1,x_1,\ldots ,x_{q-1}\}_{k/k} = 0$ if $x_i = x_j$  
    for some $i<j$.  
    The element $\{1,x_1,\ldots ,x_{q-1}\}_{k/k}$ 
    is in the image of the homomorphism 
    \begin{equation}
    \label{eq:dlog}
    \dlog:K(k;\overbrace{\Gm,\ldots ,\Gm}^{q-1}) \to K(k;\Ga,\overbrace{\Gm,\ldots ,\Gm}^{q-1})
    \end{equation}
    defined by $\{y_1,\ldots ,y_{q-1}\}_{E/k} \mapsto \{1,y_1,\ldots ,y_{q-1}\}_{E/k}$. 
    From the isomorphism $K(k;\Gm,\ldots ,\Gm) \isomto K^M_{q-1}(k)$, 
    we have $\{1, x_1,\ldots ,x_{q-1}\}_{k/k} = 0$. 
    For the Cathelineau relation (=the second equality in (\ref{eq:relation})),  
    we may assume $x\neq 1$ 
    and $x_i\neq 1$ for all $i$. 
    Consider the function field $K = k(\mathbb{P}^1_k) = k(t)$. 
    Take 
    \[
      f = t^{-1}, g = t, g_1 = \frac{(t-x)(t-(1-x))}{t(t-1)},\ 
      \mathrm{and}\ g_k = x_{k}\ \mathrm{for}\ 2 \le k \le q-1.
    \]
    For any $P\neq P_{\infty} := $ infinite place in $K$, 
    $g \in \wh{\O}_{\mathbb{P}^1,P}$ and we have
    \begin{align*}
      \{g(P),&\dP(g_{1}, f),g_2(P),\ldots,g_{q-1}(P)\}_{k/k}\\ 
      &= \begin{cases}
        \{x,x, x_2,\ldots ,x_{q-1}\}_{k/k} ,& P = (t-x),\\
        \{1-x,1-x, x_2,\ldots ,x_{q-1}\}_{k/k} ,& P = (t-(1-x)),\\
        0,& \mathrm{otherwise}.    
        \end{cases}
    \end{align*}
    For $P = P_{\infty}$, 
    $v_P(1-f) = 0, v_P(g) = -1$, $m_P(g_1) = v_P(1-g_1) \ge 2$ and $m_P(g_i) = 0$ for all $i\ge 2$. 
    Hence they satisfy the condition (\ref{eq:modulus}) for $m = 1$. 
    We also have
    $\{\dP(g, f), g_1(P), \ldots, g_{q-1}(P)\}_{k/k} = 0$. 
    The relation of type (WR) gives the desired result.

    \sn
    \textit{Definition of $\psi$}: 
    Now we define a homomorphism 
    \[
    \psi: K(k;\Ga,\Gm,\ldots,\Gm) \to \Omega_k^{q-1}
    \] 
    by 
    \[
    \psi(\{x,x_1,\ldots,x_{q-1}\}_{E/k}) :=  
      \Tr_{E/k}(x \dlog x_1  \cdots  \dlog x_{q-1}), 
    \]
    where $\Tr_{E/k}:\Omega_{E}^{q-1} \to \Omega_k^{q-1}$ 
    is the trace map (Thm.\ \ref{thm:trace}). 
    To show that $\psi$ is well-defined, 
    first consider the relations of type (PF). 
    For a finite extension $E_2/E_1$, 
    and 
    if $x_{i_0} \in \Gm(E_2), x\in \Ga(E_1), x_i \in \Gm(E_1)$ ($i\neq i_0)$, 
    by the properties of the trace map 
    (Thm.\ \ref{thm:trace}), 
    we have
    \begin{align*}
    \psi(\{x, x_1,\ldots ,x_{q-1}\}_{E_2/k})
     &= \Tr_{E_2/k}(x\dlog x_1 \cdots \dlog x_{i_0}  \cdots  \dlog x_{q-1}) \\
    &= \Tr_{E_1/k}\circ \Tr_{E_2/E_1}(x\dlog x_1 \cdots  \dlog x_{i_0}  \cdots \dlog x_{q-1}) \\
    &\stackrel{(\ast)}{=} \Tr_{E_1/k}(x \dlog x_1 \cdots \dlog N_{E_2/E_1}x_{i_0} \cdots \dlog x_{q-1})\\
    &= \psi(\{x,x_1,\ldots , N_{E_2/E_1}x_{i_0},\ldots ,x_{q-1}\}_{E_1/k}).
    \end{align*}
    Here, for the third equality $(\ast)$ above, we used the compatibility of 
    the trace map and the norm map (\Cf \cite{Kunz}, Sect.\ 16) which is given by 
    the following commutative diagram: 
    \[
      \xymatrix@C=15mm{
      E_2^{\times} \ar[d]_{N_{E_2/E_1}}\ar[r]^-{\dlog} & \Omega^{1}_{E_2}\ar[d]^{\Tr_{E_2/E_1}}\\
      E_1^{\times} \ar[r]^-{\dlog} & \Omega_{E_1}^{1},
      }
    \]
    Similarly, 
    if $x \in \Ga(E_2)$ and $x_i \in \Gm(E_1)$, 
    we have 
    \begin{align*}
    \psi(\{x,x_1,\ldots ,x_{q-1}\}_{E_2/k}) &= \Tr_{E_2/k}(x\dlog x_1 \cdots \dlog x_{q-1}) \\
    &= \Tr_{E_1/k}\circ \Tr_{E_2/E_1}(x\dlog x_1 \cdots \dlog x_{q-1}) \\
    & = \Tr_{E_1/k}( (\Tr_{E_2/E_1}x)\dlog x_1 \cdots \dlog x_{q-1})\\
    & = \psi(\{\Tr_{E_2/E_1}x, x_1,\ldots ,x_{q-1}\}_{E_1/k}).
    \end{align*}
    Next, for the relation (WR), 
    let $K = k(C)$ be the function field of a curve $C$ over $k$. 
    Take $f\in K^{\times}, g_0 \in\Ga(K), g_i\in \Gm(K)$ ($1 \le i \le q-1)$, 
    as in Definition \ref{def:Somekawa}. 
    For each $P\in C_0$, 
    if $i(P) \neq 0$, 
    then Proposition \ref{prop:res} and Lemma \ref{lem:tameres} imply 
    \begin{align*}
    &\psi(\{g_0(P), g_1(P), \ldots, 
        \dP(g_{i(P)}, f),  \ldots ,g_{q-1}(P)\}_{k(P)/k}) \\
    &\quad = \Tr_{\kP/k}(g_0(P)\dlog 
      \{ g_1(P),  \ldots, \dP(g_{i(P)},f), \ldots, g_{q-1}(P)\})\\
    &\quad = \Tr_{\kP/k}(g_0(P)\dlog \dP(\{g_1, \ldots, g_{q-1},f\}))\\
    &\quad = \Tr_{\kP/k}(g_0(P)\dP(\dlog g_1 \cdots \dlog g_{q-1} \dlog f))\\
    &\quad = \Tr_{\kP/k}\circ \dP(g_0\dlog g_1\cdots \dlog g_{q-1} \dlog f).
    \end{align*}
    Now we consider the case $i(P) = 0$.  
    We show the following equation: 
    \begin{equation}
    \label{eq:res}
      \dP(g_0\dlog(g_1(P)) \cdots \dlog(g_{q-1}(P)) \dlog f) = \dP(g_0\dlog g_1\cdots \dlog g_{q-1}\dlog f). 
    \end{equation}
    If $v_P(g_0)\ge 0$, then for $f = ut^{v_P(f)}$ with $u\in \Ohat_{C,P}^{\times}$  we have 
    \begin{align*}
      \dP(g_0\dlog g_1 \cdots \dlog g_{q-1}\dlog f) 
     &= v_P(f)\dP(g_0\dlog g_1 \cdots \dlog g_{q-1}\dlog t)\\
     &= v_P(f)g_0(P)\dlog (g_1(P)) \cdots \dlog (g_{q-1}(P)) \\
     &= \dP(g_0\dlog (g_1(P)) \cdots \dlog (g_{q-1}(P))\dlog f). 
    \end{align*}
    Next we assume $v_P(g_0)<0$.   
    Put $m_i = m_P(g_i)$ for $1 \le i \le q-1$. 
    When $v_P(1 - f) \le 0$,  
    the condition (\ref{eq:modulus}) implies that there exists 
    $i$ such that $m_P(g_i) = m_i > 0$. 
    This assures that $\dlog (g_i(P)) = 0$ and thus the left hand side of (\ref{eq:res}) is $= 0$. 
    Without loss of generality, 
    we may assume $m_1,\ldots, m_k >0$ and $m_{k+1} = \cdots = m_{q-1} = 0$ for some $k\ge 1$. 
    Putting $g_i = 1+ u_it^{m_i}$ for $1\le i \le k$ with $u_i \in \Ohat_{C,P}^{\times}$, 
    we have 
    \begin{align*}
      &\dlog g_1 \cdots \dlog g_{q-1}\dlog f\\ 
      &\quad  = \frac{t^{m_1-1}(u_1m_1dt + tdu_1)}{g_1} \cdots 
        \frac{t^{m_k-1}(u_km_kdt + tdu_k)}{g_k}\frac{d g_{k+1}}{g_{k+1}} 
        \cdots \frac{d g_{q-1}}{g_{q-1}}\left(\frac{du}{u} + v_P(f)\frac{dt}{t}\right).
    \end{align*}
    Thus $t^{\sum_{i=1}^k m_i - 1}$ divides $\dlog g_1\cdots \dlog g_{q-1}$.\ 
    The condition (\ref{eq:modulus}) gives inequalities  
    \begin{align*}
      v_P(g_0) + \sum_{i=1}^{q-1} m_i - 1 
      & \ge v_P(g_0) + \sum_{i=1}^{q-1} m_i + v_P(1-f) -1\\
      & \ge - v_P(g_0) -1\\
      & \ge 0.
    \end{align*}
    Therefore the right hand side of (\ref{eq:res}) is also equal to $0$. 
    In the case of $m := v_P(1-f) > 0$ also, 
    may assume $m_1,\ldots, m_k >0$ and $m_{k+1} = \cdots = m_{q-1} = 0$. 
    When such $k$ does not exist (that is $m_i = 0$ for all $i$), 
    the condition (\ref{eq:modulus}), that is,   
    $2v_P(g_0) + m \ge 0$ assures 
    $\dP(g_0\dlog f) = 0$ and the equation (\ref{eq:res}) holds. 
    Putting 
    $g_i = 1+ u_it^{m_i}$ for $1\le i \le k$ with $u_i \in \Ohat_{C,P}^{\times}$ 
    and $f = 1 + ut^{m}$ with $u\in \Ohat_{C,P}^{\times}$, we have 
    \begin{align*}
      &\dlog g_1 \cdots \dlog g_{q-1}\dlog f\\ 
      &\quad  = \frac{t^{m_1-1}(u_1m_1dt + tdu_1)}{g_1} \cdots 
        \frac{t^{m_k-1}(u_km_kdt + tdu_k)}{g_k}\frac{d g_{k+1}}{g_{k+1}} \cdots 
        \frac{d g_{q-1}}{g_{q-1}} \frac{t^{m-1}(u m dt + tdu)}{f}.
    \end{align*}
    Thus $t^{\sum_{i=1}^{k}m_i + m - 1}$ divides $\dlog g_1\cdots \dlog g_{q-1}\dlog f$ 
    and it gives inequalities  
    \begin{align*}
    v_P(g_0) + \sum_{i=1}^{q} m_i + m - 1 
      &\ge - v_P(g_0) -1\\ 
      &\ge 0.
    \end{align*}
    Both sides of (\ref{eq:res}) are $0$. 
    Using the equation (\ref{eq:res}), we obtain  
    \begin{align*}
    &\psi(\{\dP(g_0,f), g_1(P), \ldots ,g_{q-1}(P)\}_{k(P)/k}) \\
    &\quad = \Tr_{\kP/k}(\dP(g_0\dlog f) \dlog g_1(P) \cdots \dlog g_{q-1}(P))\\
    &\quad = \Tr_{\kP/k}\circ \dP(g_0 \dlog g_1(P) \cdots \dlog g_{q-1}(P) \dlog f)\\
    &\quad = \Tr_{\kP/k}\circ \dP(g_0 \dlog g_1\cdots \dlog g_{q-1} \dlog f).
    \end{align*}
    From the residue theorem 
    (Thm.\ \ref{thm:res}), 
    we have 
    \[
    \psi\left(\sum_{P\in C_0} \{g_0(P), g_1(P), \ldots, 
        \dP(g_{i(P)}, f),  \ldots ,g_{q-1}(P)\}_{k(P)/k}\right) = 0.
    \]
    Thus $\psi$ is well-defined. 

    \sn
    \textit{Proof of the bijection}: 
    It is easy to see that $\psi \circ \phi = \Id$, where $\Id$ is the identity map on 
    $\Omega_k^{q-1}$. Hence we show that $\phi$ is surjective. 
    Take an element 
    $\{x,x_1,\ldots ,x_{q-1}\}_{E/k}$ in $K(k;\Ga,\Gm,\ldots,\Gm)$. 
    The trace map $\Omega_E^{q-1} \to \Omega_k^{q-1}$ 
    is defined from the trace map 
    $\Tr_{E/k}:E\to k$ by 
    identifying 
    the isomorphism $E\otimes_k \Omega_k^{q-1} \simeq \Omega_E^{q-1}$ 
    (Thm.\ \ref{thm:trace}).  
    Therefore 
    $\Tr_{E/k}(x \dlog x_1 \cdots \dlog x_{q-1})  
    = \sum_i \Tr_{E/k}(\xi_i) \dlog z_{i,1} \cdots \dlog z_{i,q-1}$ 
    for some $\xi_i \in E$ and $z_{i,j} \in k^{\times}$. 
    The property (PF) implies the equality 
    \[
      \{\Tr_{E/k}\xi_i, z_{i,1},\ldots , z_{i,q-1}\}_{k/k}  
        = \{\xi_i, z_{i,1},\ldots ,z_{i,q-1}\}_{E/k}.
    \] 
    Therefore, defining  
    $\phi_E :\Omega_E^{q-1} \to K(E;\Ga,\Gm,\ldots ,\Gm)$ 
    by 
    $\xi \dlog \eta_1 \cdots \dlog \eta_{q-1} \mapsto \{\xi, \eta_1 ,\ldots , \eta_{q-1}\}_{E/E}$ 
    we have 
    \begin{align*}
      &\phi(\Tr_{E/k}(x \dlog x_1 \cdots \dlog x_{q-1}))\\ 
      &\qquad =\sum_i \phi (\Tr_{E/k}\xi_i\dlog z_{i,1} \cdots \dlog z_{i,q-1}) \\
      &\qquad  = \sum_i N_{E/k} \circ \phi_E(\xi_i \dlog z_{i,1} \cdots \dlog z_{i,q-1}) \\
      &\qquad  = N_{E/k} \circ \phi_E(x \dlog x_1 \cdots \dlog x_{q-1}) \\
      &\qquad  = \{x,x_1,\ldots ,x_{q-1}\}_{E/k}.
    \end{align*}
    Thus $\phi$ is surjective and we obtain the assertion. 
\end{proof}

The isomorphism $E\otimes_k \Omega_k^{q-1} \simeq \Omega_E^{q-1}$ 
appeared in the above proof 
is comparable with the following fact on the Milnor $K$-groups 
(\cite{BT73}, Chap.\ I, Cor.\ 5.3; see also the proof of Thm.\ 1.4 in \cite{Som90}): 
For a prime number $p$, we assume that 
every finite extension of a field $k$ is of degree $p^r$ for some $r$. 
Then $K_q^M(E)$ is generated by elements of the form 
$\{x, y_1,\ldots ,y_{q-1}\}$, where $x \in E^{\times}$ and 
$y_i \in \kt$.  

If $\Char(k) = 0$, 
for a finite truncation set $S$, 
we have 
an isomorphism  
\begin{equation}
\label{eq:Gh}
  \Gh = (\Gh_{s})_{s\in S}:\WSOmegak^{q-1} \isomto (\Omega_k^{q-1})^{S}
\end{equation}
given by $\Gh_s(\omega) = F_s(\omega)|_{\{1\}}$ 
from Proposition \ref{prop:dW} (by taking a prime $p\not\in S$).
Reducing to the case of $\Ga$, 
we obtain the following corollary.

\begin{corollary}
    \label{cor:WS}
    Let $k$ be a perfect field with $\Char(k) \neq 2$ 
    and $S$ a finite truncation set.
    Then we have 
    \[
      K(k;\WS,\overbrace{\Gm,\ldots,\Gm}^{q-1}) \isomto \WSOmega_k^{q-1}.
    \]
\end{corollary}
\begin{proof}
    We may assume $q>1$ (Lem.\ \ref{lem:isom}). 
    The left becomes trivial if $\Char(k)> 0$ (Lem.\ \ref{lem:trivial}). 
    It is enough to show $\WSOmega_k^{q-1} =0$ in this case. 
    By Proposition \ref{prop:dW} 
    this reduces to the well-known result 
    $W_m\Omega_k^{q-1}=0$ for some $m$ (\cite{Ill79}, Chap.\ I, Prop.\ 1.6). 
    Now we assume $\Char(k) = 0$. 
    The ghost map (\ref{eq:ghost}) gives an isomorphism $\gh = (\gh_s)_{s\in S}:\WS \isomto \Ga^{\# S}$. 
    We show that the map 
    \[
     \gamma: K(k;\WS,\Gm,\ldots , \Gm) \to K(k;\Ga, \Gm,\ldots , \Gm) 
    \]
    given by 
    $\{w,x_1,\ldots ,x_{q-1}\}_{E/k} \mapsto \{\gh_s(w),x_1,\ldots , x_{q-1}\}_{E/k}$
    is well-defined for $s\in S$. 
    For the relation (PF), 
    since the ghost map $\gh_s:\WS \to \Ga$ 
    commutes with $j_{\ast}$ and $j^{\ast}$ 
    for some finite extension $j:E_1 \inj E_2$ over $k$ (\cite{Rue07}, Appendix A) 
    the map $\gamma$ annihilates the elements of the form (PF). 
    For the relation (WR), 
    let $K = k(C)$ be the function field of 
    a curve $C$ over $k$. 
    For $f\in \Kt, g_0\in \WS(K), g_i\in \Gm(K)$ and $P\in C_0$,  
    if 
    $i(P) \neq 0$ then 
    \[
    \gamma (\{g_0(P), \ldots, \dP(g_{i(P)}, f) , \ldots , g_{q-1}(P)\}_{k(P)/k}) 
    = \{\gh_s(g_0)(P), \ldots, \dP(g_{i(P)}, f) , \ldots , g_{q-1}(P)\}_{k(P)/k}. 
    \]
    Consider the case of $i(P)  = 0$. 
    Immediately, the equality 
    \begin{align*}
      \gamma (\{\dP(g_0, f), g_1(P) , \ldots, g_{q-1}(P)\}_{k(P)/k})  
      &= \{v_P(f)\gh_s(g_0)(P), g_1(P) , \ldots, g_{q-1}(P)\}_{k(P)/k} \\
      &= \{\dP(\gh_s(g_0), f), g_1(P) ,\ldots , g_{q-1}(P)\}_{k(P)/k} 
    \end{align*}
    holds if $g_0 \in \WS(\Ohat_{C,P})$ from (\ref{eq:residue}). 
    If $g_0 \not\in \WS(\Ohat_{C,P})$ and $v_P(\gh_s(g_0)) <0$, then 
    the condition (\ref{eq:modulus}) gives 
    \begin{align*}
      &2v_P(\gh_s(g_0)) + \sum_{i=1}^{q-1}m_P(g_i) + v_P(1-f) \\
      &\quad  \ge  (m+1)v_P(\gh_s(g_0)) + \sum_{i=1}^{q-1}m_P(g_i) + v_P(1-f) \\
      &\quad \ge 0,
    \end{align*}
    where $m := \max\{s \in S\}$. 
    By Proposition \ref{prop:res} (c) and  
    $\Gh_s = F_s(-)|_{\{1\}} = \gh_{\{1\}}\circ F_s$,  
    we obtain 
    \begin{align*}
    \gh_s(\dP(g_0, f)) 
      & = \gh_s\circ \Res_{S}(g_0\dlog[f])\\
      & = \Res_{\{1\}}\circ \Gh_s(g_0\dlog[f]). 
    \end{align*}
    By the very definition of Witt complexes (\cite{Rue07}, Def.\ 1.4), 
    we have $F_s d [f] = [f]^{s-1}d[f]$ and thus 
    \begin{align*}
    \Gh_s(g_0\dlog[f]) 
      &= \Gh_s(g_0)\Gh_s(\dlog [f])\\
      &= \gh_{\{1\}}\circ F_s(g_0) \frac{\gh_{\{1\}}\circ F_s d[f]}{\gh_{s} [f]}\\
      &=  \gh_s(g_0)\frac{[f]^{s-1} d[f]}{[f]^s} \\
      &= \gh_s(g_0)\dlog [f].
    \end{align*}
    Therefore, 
    \begin{align*}
    \gamma (\{\dP(g_0, f), g_1(P) , \ldots, g_{q-1}(P)\}_{k(P)/k})  
      &= \{\gh_s(\dP(g_0, f)), g_1(P) , \ldots, g_{q-1}(P)\}_{k(P)/k} \\
      &= \{\Res_{\{1\}}(\gh_s(g_0)\dlog[f]), g_1(P) ,\ldots , g_{q-1}(P)\}_{k(P)/k} \\
      &= \{\dP(\gh_s(g_0), f), g_1(P) ,\ldots , g_{q-1}(P)\}_{k(P)/k}.
    \end{align*}
    The map $\gamma$ is well-defined and we obtain an isomorphism
    \[
      K(k;\WS,\Gm,\ldots , \Gm) \isomto K(k;\Ga,\Gm,\ldots , \Gm)^S.
    \]
    The assertion now follows from the isomorphism (\ref{eq:Gh}) and Theorem \ref{thm:main}.
\end{proof}

\section{Mixed $K$-groups}
\label{sec:mix}
For an equidimensional variety $X$ over a perfect field $k$, 
let $\CH_0(X)$ be the Chow group of $0$-cycles on $X$. 
Putting $X_E := X\otimes_k E$ for an extension $E/k$,  
the assignments 
$\cCH_0(X): E \mapsto \CH_0(X_E)$ 
give a Mackey functor 
by the pull-back $j^{\ast}:\CH_0(X_{E_2})\to \CH_0(X_{E_1})$ 
and the push-forward $j_{\ast}:\CH_0(X_{E_1}) \to \CH_0(X_{E_2})$ 
for any finite field extension $j:E_1\inj E_2$ over $k$ 
(\Cf \cite{Ful98}, Chap.\ 1). 

\begin{definition} 
    \label{def:mix}
    Let $r,q\ge 0$ be integers, 
    $X_1,\ldots ,X_r$ smooth projective varieties over $k$, 
    and $G_1,\ldots ,G_q$ split semi-abelian varieties over $k$. 
    The \textit{mixed $K$-group} 
    \[
    K(k; \WS, G_1,\ldots ,G_q, \cCH_0(X_1),\ldots ,\cCH_0(X_r))
    \] 
    is defined by the quotient of the Mackey product 
    \[
      \left(\WS \otimesM G_1 \otimesM \cdots 
        \otimesM G_q\otimesM \cCH_0(X_1)\otimesM \cdots \otimesM \cCH_0(X_r)(k)\right)/R,  
    \] 
    where $R$ is the subgroup generated by the following elements: 

    \sn
    (WR) Let $K = k(C)$ be the function field of 
    a projective smooth curve $C$ over $k$. 
    Put $G_0 := \WS$. 
    For $f\in \Kt, x_j \in \CH_0((X_j)_K)$ 
    and $g_i\in G_i(K)$, 
    assume that for each closed point $P$ in $C$ there exists 
    $i(P)$ such that $g_i \in G_i(\Ohat_{C,P})$ for all $i\neq i(P)$. 
    If $g_0 \not \in \WS(\Ohat_{C,P})$ and $q\ge 1$, we further assume, 
    for any $s\in S$ with $v_P(\gh_s(g_0)) < 0$,   
    \begin{align}
    \label{eq:modulus2}
       (m+1)v_P(\gh_s(g_0)) + \sum_{i=1}^{q}m_P(g_i) + v_P(1-f) \ge 0,
    \end{align}
    where $m := \max\{s \in S\}$. 
    Then, the element is 
    \[ 
      \sum_{P \in C_0} g_0(P)\otimes \cdots \otimes \dP(g_{i(P)}, f)\otimes 
        \cdots \otimes g_q(P)\otimes s_P(x_1)\otimes \cdots\otimes s_P(x_r) \in R.
    \]
    Here $s_P:\CH_0((X_j)_K) \to \CH_0((X_j)_{\kP})$ is the 
    specialization map at $P$ (\Cf \cite{Ful98}, Sect.\ 20.3). 
\end{definition}
Note that this group coincides with the $K$-group defined in \cite{Akhtar04a} (and \cite{RS00}) 
in case $G_0 = \WS$ is trivial. 
This mixed $K$-group also has the structure of $\WS(k)$-module. 

\begin{lemma}
    For projective smooth varieties $X$ and $Y$ over $k$, we have 
    \[
        K(k; \WS, \cCH_0(X), \cCH_0(Y)) \isomto K(k; \WS, \cCH_0(X\times Y)).
    \] 
\end{lemma}
\begin{proof}
    \textit{Definition of $\psi$}: 
    By \cite{RS00}, Theorem 2.2, 
    there is a canonical isomorphism 
    \[
      \psi^{RS}_k: K(k; \cCH_0(X), \cCH_0(Y)) \isomto  K(k; \cCH_0(X\times Y)) 
    \isomto \CH_0(X\times Y)
    \]
    defined by $\psi^{RS}_k(\{x,y\}_{E/k}) 
      := N_{E/k}((p_X)^{\ast}x \cap (p_Y)^{\ast}y)$, 
    where  $\cap$  is the intersection product (\Cf \cite{Ful98}, Chap.\ 8), 
    $p_X:X\times Y \to X,p_Y:X\times Y \to Y$ are the projection maps, 
    and 
    $N_{E/k} = j_{\ast}:\CH_0((X\times Y)_E) \to \CH_0(X\times Y)$ is the push-forward 
    along $j:(X\times Y)_E \to X\times Y$ (\Cf \cite{Ful98}, Sect.\ 1.4). 
    We denote by $\psi$ the composition 
    \[
      \WS \otimesM \cCH_0(X) \otimesM \cCH_0(Y)(k) \surj  
      \WS\otimesM \cCH_0(X\times Y) (k) \surj K(k;\WS, \cCH_0(X\times Y)). 
    \] 
    Precisely, the map $\psi$ is given by 
    \begin{align*}
    \psi(\{w, x,y\}_{E/k}) & = \psi(\{w, \{x,y\}_{E/E}\}_{E/k})\\
     &= \{w, \psi^{RS}_{E}(\{x,y\}_{E/E})\}_{E/k}\\
     &= \{w, (p_X)^{\ast}x \cap (p_Y)^{\ast}y\}_{E/k}.
    \end{align*}
    Now we show that the map $\psi$ annihilates the elements of the form (WR) in 
    $K(k; \WS, \cCH_0(X), \cCH_0(Y))$. 
    Let $K = k(C)$ be the function field of a projective smooth curve $C$ 
    over $k$, 
    $f\in K^{\times}$, $w \in \WS(K)$, $x\in \CH_0(X_K)$ and $y\in \CH_0(Y_K)$. 
    Since the specialization map is compatible with pull-backs 
    (\cite{Ful98}, Prop.\ 20.3 (a)), 
    for any closed point $P\in C_0$, we have
    \begin{align*}
      &\psi(\{\dP(w,f), s_P(x), s_P(y)\}_{\kP/k}) \\
      &\quad = \{\dP(w,f),(p_X)^{\ast}s_P(x) \cap (p_Y)^{\ast}s_P(y)\}_{\kP/k}\\
      &\quad = \{ \dP(w,f), s_P ((p_X)^{\ast}x \cap (p_Y)^{\ast}(y))\}_{\kP/k}.
    \end{align*}
    Thus (WR) in $K(k;\WS, \cCH_0(X \times Y))$ implies the assertion. 

    \sn
    \textit{Definition of $\phi$}: 
    Now we define 
    \[
      \phi:K(k;\WS, \cCH_0(X\times Y)) \longrightarrow K(k;\WS, \cCH_0(X), \cCH_0(Y)).
    \]
    For any finite field extension $E/k$, 
    the Chow group $\CH_0((X\times Y)_E)$ is generated by 
    closed points $P$ in $(X\times Y)_E$. 
    The map $\phi$ is defined by, for a closed point $P$ in $(X\times Y)_E$,  
    \[
      \phi(\{w, [P]\}_{E/k}) := \{w, (p_X)_{\ast}[P], (p_Y)_{\ast}[P]\}_{E/k}.
    \] 
    Note that the inverse map $\phi^{RS}_k$ of $\psi^{RS}_k$ 
    is given  by $\phi^{RS}_k([P]) = \{(p_X)_{\ast}[P], (p_Y)_{\ast}[P]\}_{E/k}$. 
    First we show that this map $\phi$ is well-defined.  
    Let  $j:E_1 \inj E_2$ be a finite extensions over $k$. 
    For $w \in \WS(E_2)$ and a closed point $P$ in $(X\times Y)_{E_1}$, 
    we have
    \begin{align*}
      \phi(\{w, j^{\ast}[P]\}_{E_2/k})  
      &= \{w, (p_X)_{\ast}j^{\ast}[P], (p_Y)_{\ast}j^{\ast}[P]\}_{E_2/k}\\
       &= \{w, j^{\ast}(p_X)_{\ast}[P], j^{\ast}(p_Y)_{\ast}[P]\}_{E_2/k}\\
      & \stackrel{(\ast)}{=} \{\Tr_{E_2/E_1}w, (p_X)_{\ast}[P], (p_Y)_{\ast}[P]\}_{E_1/k}\\
      &= \phi(\{\Tr_{E_2/E_1}w, [P]\}_{E_1/k}).
    \end{align*} 
    For the equality $(\ast)$, here we used the projection formula (PF) 
    in $K(k; \WS, \cCH_0(X), \cCH_0(Y))$. 
    For $w \in \WS(E_1)$ and a closed point $P$ in $(X\times Y)_{E_2}$, 
    we show $\phi(\{j^{\ast}(w), [P]\}_{E_2/k}) = \phi(\{w, j_{\ast}[P]\}_{E_1/k})$. 
    From the commutative diagram
    \[
      \xymatrix{
        \CH_0((X\times Y)_{E_2}) \ar[d]_{j_{\ast} 
          = N_{E_2/E_1}} \ar[r]^-{\phi^{RS}_{E_2}} & K(E_2;\cCH_0(X), \cCH_0(Y)) \ar[d]^{N_{E_2/E_1}} \\
        \CH_0((X\times Y)_{E_1}) \ar[r]^-{\phi^{RS}_{E_1}} & K(E_1;\cCH_0(X), \cCH_0(Y)) 
       }
    \]
    we have 
    \begin{align*}
    \{(p_X)_{\ast}[P], (p_Y)_{\ast}[P]\}_{E_2/E_1} 
    &= N_{E_2/E_1}(\{(p_X)_{\ast}[P], (p_Y)_{\ast}[P]\}_{E_2/E_2}) \\
    &= N_{E_2/E_1}\circ \phi_{E_2}^{RS}([P]) \\
    &= \phi_{E_1}^{RS}\circ N_{E_2/E_1} ([P]) \\
    &= \{(p_X)_{\ast}j_{\ast}[P], (p_Y)_{\ast}j_{\ast}[P]\}_{E_1/E_1}. 
    \end{align*}
    Recall that the $\WS(k)$-module structure on $K(k;\WS, \cCH_0(X), \cCH_0(Y))$ 
    is given by $v\{w,x, y\}_{E/k} = \{j^{\ast}(v)w, x,y\}_{E/k}$ 
    for $v\in \WS(k)$, and a symbol $\{w, x, y\}_{E/k}$, 
    where $j:k\inj E$. Hence 
    \begin{align*}
    \phi(\{j^{\ast}w, [P]\}_{E_2/k}) 
     &= \{j^{\ast}w, (p_X)_{\ast}[P], (p_Y)_{\ast}[P]\}_{E_2/k}\\
     &= N_{E_1/k}(\{j^{\ast}w, (p_X)_{\ast}[P], (p_Y)_{\ast}[P]\}_{E_2/E_1})\\
     &= N_{E_1/k}(w\{1, (p_X)_{\ast}[P], (p_Y)_{\ast}[P]\}_{E_2/E_1})\\
     &\stackrel{(\ast)}{=} N_{E_1/k}(w\{1, (p_X)_{\ast}j_{\ast}[P], (p_Y)_{\ast}j_{\ast}[P]\}_{E_1/E_1})\\
     &= \{w, j_{\ast}(p_X)_{\ast}[P], j_{\ast}(p_Y)_{\ast}[P]\}_{E_1/k}\\
     &= \{w, (p_X)_{\ast}j_{\ast}[P], (p_Y)_{\ast}j_{\ast}[P]\}_{E_1/k}\\
     &= \phi(\{w, j_{\ast}[P]\}_{E_1/k}).
    \end{align*}
    Here, the equality ($\ast$) holds since the both of the symbols 
    $\{1, (p_X)_{\ast}[P], (p_Y)_{\ast}[P]\}_{E_2/E_1}$ and 
    $\{1, (p_X)_{\ast}j_{\ast}[P], (p_Y)_{\ast}j_{\ast}[P]\}_{E_1/E_1}$ 
    are in the image of the following map like $\dlog$ in (\ref{eq:dlog})
    \[
      K(E_1; \cCH_0(X), \cCH_0(Y)) \to K(E_1; \WS, \cCH_0(X), \cCH_0(Y))
    \]
    which is defined by $\{x,y\}_{E_1'/E_1} \mapsto \{1,x,y\}_{E_1'/E_1}$. 
    Next, we show that the map $\phi$ takes the elements of the form (WR) 
    to zero. 
    Let $K = k(C)$ be the function field of a projective smooth curve $C$ 
    over $k$, 
    $f\in K^{\times}$, $w \in \WS(K)$ and $\xi$ a closed point 
    in $(X\times Y)_K$. 
    Since the specialization map is compatible with push-forward 
    (\cite{Ful98}, Prop.\ 20.3 (b)), 
    for any closed point $P\in C_0$, we have
    \begin{align*}
      \phi(\{\dP(w,f), s_P[\xi]\}_{\kP/k}) 
      & = \{\dP(w,f), (p_X)_{\ast}s_P[\xi], (p_Y)_{\ast}s_P[\xi]\}_{\kP/k}\\
      &= \{\dP(w,f), s_P\circ (p_X)_{\ast}[\xi], s_P\circ (p_Y)_{\ast}[\xi] \}_{\kP/k}.
    \end{align*}
    Thus the assertion follows from 
    (WR) in $K(k;\WS, \cCH_0(X), \cCH_0(Y))$. 

    \sn
    \textit{Proof of the bijection}: 
    For any symbol $\{w, [P_X], [P_Y]\}_{E/k}$ in $K(k;\WS,\cCH_0(X), \cCH_0(Y))$ 
    which is given by closed points $P_X$ in $X_E$ and $P_Y$ in $Y_E$,  
    there exists a finite field extension $L/E$  such that 
    $L$ is a Galois extension over $E(P_X)$ and $E(P_Y)$. 
    Then $(P_X\times_E P_Y)_L$ decomposes of $L$-rational points. 
    The projection formula (PF) implies 
    \begin{align*}
      \{w, [P_X], [P_Y]\}_{E/k} &= \{\Tr_{L/E}\wt{w}, [P_X], [P_Y]\}_{E/k} \\
     &= \{ \wt{w}, j^{\ast}[P_X], j^{\ast}[P_Y]\}_{L/k} 
    \end{align*}
    for some $\wt{w} \in \WS(L)$ and $j:E\inj L$. 
    Hence $K(k;\WS, \cCH_0(X), \cCH_0(Y))$ is generated by 
    symbols of the form $\{w, P_X, P_Y\}_{E/k}$ for a finite field extension $E/k$, 
    where $P_X$ and $P_Y$ are $E$-rational points. 
    The same holds in $K(k;\cCH_0(X\times Y))$. 
    Now it is easy to see that $\phi$ is surjective and 
    $\psi\circ \phi(\{w, [P]\}_{E/k}) = \{w, [P]\}_{E/k}$ 
    for an $E$-rational point $P$ in $(X\times Y)_E$. 
    The bijectivity follows from it. 
\end{proof}

\begin{lemma}
    \label{lem:mix}
    $\mathrm{(i)}$ If $X = \Spec k$, then 
    \[
    K(k; \WS, \cCH_0(\Spec k)) \isomto K(k; \WS) \simeq \WS(k). 
    \]

    \sn 
    $\mathrm{(ii)}$ If $\Char(k) = p>0$ and 
    $X$ is a projective smooth and geometrically connected curve over $k$ 
    with $X(k) \neq \emptyset$, 
    then 
    \[
      K(k; \WS, \cCH_0(X)) \isomto K(k; \WS) \simeq \WS(k).
    \]
\end{lemma}
\begin{proof}
    The assertion (i) follows from 
    $\CH_0(\Spec E) \simeq \Z$ for an 
    extension $E/k$. 
    For the assertion (ii), we consider 
    the kernel $\cA_0(X)$ of the map 
    $\deg:\cCH_0(X) \to \Z$ of Mackey functors given by the degree map. 
    Note that the specialization map induces a map on $\cA_0(X)$. 
    By replacing $\cCH_0$ with $\cA_0$ in the appropriate instances, 
    we can define the mixed $K$-group $K(k; \WS, \cA_0(X))$ 
    as a quotient of $\WS\otimesM \cA_0(X)(k)$. 
    Recall that 
    $\Z$ is the unit of the Mackey product. 
    The assumption $X(k)\neq \emptyset$ gives 
    a split short exact sequence 
    \[
      0 \to K(k;\WS,\cA_0(X)) \to K(k;\WS,\cCH_0(X)) \to K(k;\WS) \to 0.
    \]
    Because of $K(k;\WS) \simeq \WS(k)$ (Lem.\ \ref{lem:isom}), 
    it is enough to show $\WS\otimesM \cA_0(X)(k) = 0$. 
    By the assumption $X(k)\neq \emptyset$, 
    we have 
    $\cA_0(X) \simeq J_{X}$ as Mackey functors, 
    where $J_{X}$ is the Jacobian variety of $X$. 
    Therefore we obtain 
    $\WS\otimesM \cA_0(X)(k) \simeq \WS\otimesM J_{X}(k)$, 
    and the latter becomes trivial by Lemma \ref{lem:trivial}.
\end{proof}

Let $X$ be a projective smooth variety over $k$ 
and $X_i$ the subset of $X$ consisting of points $x$ 
such that the closure of $\ol{\{x\}}$ is of dimension $i$.
On Milnor $K$-groups, there is a canonical homomorphism 
\[
  \d : \bigoplus_{\xi \in X_1}K_{2}^M(k(\xi)) \to \bigoplus_{P\in X_0}k(P)^{\times}.
\]
The map  $\d$ is defined by the boundary map of Milnor $K$-theory as follows. 
For each $\xi\in X_1$ and $P\in X_0$, let 
$\pi:Y^{N} \to Y \subset X$ be the normalization of 
the reduced scheme $Y := \ol{\{\xi\}}$ and define $\d_P^{\xi}$ by the composition
\[
  \d_P^{\xi}:K_{2}^M(k(\xi)) \xrightarrow{\oplus \d_{P^N}} 
    \bigoplus_{P^N\in \pi^{-1}(P)} k(P^{N})^{\times} 
    \xrightarrow{\sum_{P^N}{N_{k(P^N)/k(P)}}}\kP^{\times},
\]
where $\d_{P^N}$ is the tame symbol and $P^N$ runs over all closed points of $Y^N$ lying over $P$. 
The map $\d$ is defined by taking the sum of these homomorphisms $\d^{\xi}_P$. 
Using this, the mixed $K$-group associated to $\cCH_0(X)$ and $\Gm$'s 
has the following description. 
(In fact, Akhtar \cite{Akhtar04a} gives a description of 
$K(k; \overbrace{\Gm,\ldots , \Gm}^q, \cCH_0(X))$ more generally. 
But here we only refer the case of $q=1$.)

\begin{theorem}[\cite{Akhtar04a}, Thm.\ 5.2, 6.1]
    \label{thm:SK1}
    Let $X$ be a projective smooth variety over $k$. 
    There is a canonical isomorphism 
    \[
      K(k; \Gm, \cCH_0(X)) \isomto 
    \Coker\left(\d:\bigoplus_{\xi\in X_1} K_{2}^M(k(\xi)) \to \bigoplus_{P\in X_0}\kP^{\times}\right).
    \]
\end{theorem} 

\noindent
Note that the cokernel on the right is often denoted by $SK_1(X)$. 
By the Weil reciprocity law on the Milnor $K$-groups (\Cf \cite{BT73}, Sect.\ 5),  
one can consider a complex 
\[
  \bigoplus_{\xi \in X_1} K_2^M(k(\xi)) \onlong{\d} \bigoplus_{P\in X_0} \kP^{\times} \onlong{N} k^{\times}
\]
where $N$ is defined by 
the norm maps $N := \sum_P N_{\kP/k}$.  
We denote by $V(X) := \Ker(N)/\Im(\d)$ 
the homology group of the above complex.
In particular, if we further assume that 
$X$ is a geometrically connected curve over $k$ with a $k$-rational point, 
then the decomposition $A_0(X) \oplus \Z \simeq \CH_0(X)$ gives 
an isomorphism 
\[
K(k, \Gm, J_X) \isomto V(X),
\] 
where $J_X$ is the Jacobian variety of $X$. 

For a projective smooth variety $X$ over $k$ 
and a finite truncation set $S$,  
we define a canonical homomorphism 
\begin{equation}
\label{eq:dP+}
  \d = \d_S: \bigoplus_{\xi \in X_1}\WS\Omega^{1}_{k(\xi)} \to \bigoplus_{P\in X_0}\WS({k(P)}).
\end{equation}
For each $\xi\in X_1$ and $P\in X_0$, let 
$\pi:Y^N \to Y  \subset X$ be the normalization of $Y :=\ol{\{\xi\}}$
and define $\d_P^{\xi}$ by the composition
\[
  \d_P^{\xi}:\WS\Omega^{1}_{k(\xi)} 
    \xrightarrow{\oplus\d_{P^N}} \bigoplus_{P^N\in \pi^{-1}(P)} \WS({k(P^N)}) 
    \xrightarrow{\sum_{P^N}{\Tr_{k(P^{N})/k(P)}}}\WS({\kP}),
\]
where $\d_{P^N}$ is the residue map (\ref{def:res}). 
The map $\d$ is defined by taking the sum of these homomorphisms $\d^{\xi}_P$.

\begin{theorem}
    \label{thm:SK1+}
    Let $X$ be a projective smooth variety over $k$.  
    Then, there is a canonical isomorphism
    \[
    \phi :\Coker \left(\d:\bigoplus_{\xi \in X_1}\Omega^{1}_{k(\xi)} \to \bigoplus_{P\in X_0}k(P) \right) \isomto K(k; \Ga, \cCH_0(X)). 
    \]
\end{theorem}
\begin{proof}
    \textit{Definition of $\phi$}:    
    The map 
    \[
    \phi:\Coker(\d) \to K(k;\Ga,\cCH_0(X))
    \] 
    is defined by 
    $[\sum_P x_P] \mapsto \sum_P \{x_P, [P_0]\}_{\kP/k}$.  
    Here $[\sum_P x_P]$ is the class in $\Coker(\d)$ 
    represented by $\sum_P x_P \in \oplus_P \kP$ 
    and  
    $P_0$ is the canonically defined $\kP$-rational point of $X_{\kP}$ 
    which maps isomorphically to $P$ under the canonical map $X_{\kP} \to X$, 
    $[P_0]$ is the cycle in $\CH_0(X_{\kP})$ 
    represented by the point $P_0$.
    We show that this correspondence 
    is well-defined. 
    For any $\xi \in X_1$ and $g\dlog f\in \Omega_{k(\xi)}^1$, 
    the projection formula (PF) gives  
    \begin{align*}
      \phi\circ \d(g\dlog f) 
    &= \sum_{P\in X_0}\phi\circ \d_P^{\xi}(g\dlog f)\\  
    &= \sum_{P\in X_0} \sum_{P^N \in \pi^{-1}(P)} 
      \phi( \Tr_{k(P^N)/k(P)} \circ \d_{P^N}(g \dlog f))\\
    &= \sum_{P\in X_0} \sum_{P^N \in \pi^{-1}(P)} 
      \{\Tr_{k(P^N)/k(P)} (\d_{P^N}(g \dlog f)), [P_0]\}_{\kP/k}\\
    &= \sum_{P\in X_0}\sum_{P^N \in \pi^{-1}(P)}
      \{\d_{P^N}(g \dlog f), (j_{P^N/P})^{\ast}[P_0]\}_{k(P^N)/k},
    \end{align*}
    where $j_{P^N/P}$ is the canonical map $\kP\inj k(P^N)$. 
    Here, we have $(j_{P^N/P})^{\ast}[P_0] = [P_0 \times_P P^N] = [(P^N)_0]$. 
    For each closed point $P^N$ in $Y^N$, 
    the specialization map is written by 
    $s_{P^N}([\eta_0]) = [(P^N)_0]$, 
    where $[\eta_0] \in \CH_0(X_{k(Y^{N})})$ is the cycle on $X_{k(Y^N)}$
    arising from the generic point $\eta$ of $Y^N$ 
    (\cite{Ful98}, Sect.\ 20.3, see also \cite{RS00}, (2.2.2)). 
    We obtain 
    \[  
    \{\d_{P^N}(g \dlog f), (j_{P^N/P})^{\ast}[P_0]\}_{k(P^N)/k} 
      =  \{\d_{P^N}(g,f), s_{P^N}[\eta_0]\}_{k(P^N)/k}
    \]
    and thus 
    \[
      \phi\circ \d(g\dlog f) 
      = \sum_{P^N \in (Y^N)_0} \{\d_{P^N}(g,f), s_{P^N}[\eta_0]\}_{k(P^N)/k}. 
    \]
    Therefore the condition (WR) on the mixed $K$-group 
    (Def.\ \ref{def:mix}) 
    implies the well-definedness of $\phi$.

    \sn
    \textit{Definition of $\psi$}: 
    A map  
    \[
      \psi: K(k;\Ga, \cCH_0(X)) \to \Coker(\d)
    \] 
    is defined by   
    the following diagram (\Cf \cite{Bloch81}, Sect.\ 1):   
    For each finite extension $E/k$, 
    \[
    \vcenter{
    \entrymodifiers={!! <0pt, .8ex>+}
    \xymatrix@C=5mm{
      \displaystyle{\bigoplus_{\eta \in (X_E)_1} E \otimes_{\Z} E(\eta)^{\times}} 
        \ar[r] \ar[d]_-{E\otimes \dlog}&  
        \displaystyle{\bigoplus_{Q\in (X_E)_0} E}  \ar@{^{(}->}[d] \ar[r] & 
        E\otimesZ \CH_0(X_E)\ar[d] \ar[r] & 0 \\
      \displaystyle{\bigoplus_{\eta \in (X_E)_1} \Omega_{E(\eta)}^1} \ar[d]_-{\Tr}\ar[r]^-{\d_E} &
        \displaystyle{\bigoplus_{Q\in (X_E)_0}{E(Q)}} \ar[d]^-{\Tr} \ar[r] & 
        \Coker(\d_E) \ar[r]\ar[d] & 0 \\
      \displaystyle{\bigoplus_{\xi \in X_1}\Omega_{k(\xi)}^1} \ar^-{\d}[r] &
        \displaystyle{\bigoplus_{P\in X_0}\kP}\ar[r] & \Coker(\d) \ar[r] & 0 .
    }
    }
    \]
    Here, 
    the vertical maps $\Tr$ in the lower left square are given by trace maps 
    and Lemma \ref{lem:trres}. 
    Explicitly, 
    for $x\in E$ and $Q\in (X_E)_0$, 
    we have
    \[
      \psi(\{x,[Q]\}_{E/k}) = \Tr_{E(Q)/k(P)}(x) \quad \mathrm{at}\ P = j_{E/k}(Q) \in X_0, 
    \]
    where 
    $j_{E/k}: X_E \to X$ is the canonical map. 
    We show that $\psi$ factors through $K(k;\Ga, \cCH_0(X))$. 
    For any finite extension $j = j_{E_2/E_1}: E_1\inj E_2$ over $k$, 
    take $x\in E_1$ and a closed point $Q_{2}$ in $X_{E_2}$. 
    Then, $j_{\ast}[Q_2] = [E_2(Q_2):E_1(Q_1)][Q_1]$, 
    where $Q_1$ is the image of $Q_2$ in $X_{E_1}$. 
    We obtain the following equalities 
    \begin{align*}
      \psi(\{x,[Q_2]\}_{E_2/k}) 
      &=  \Tr_{E_2(Q_2)/\kP}(x) \\
      &= [E_2(Q_2):E_1(Q_1)]\Tr_{E_1(Q_1)/\kP}(x)\\
      &= \psi(\{x, j_{\ast}[Q_2]\}_{E_1/k}).
    \end{align*}
    On the other hand, 
    we assume $x\in E_2$, $Q_{1} \in (X_{E_1})_0$. 
    Denoting by $j^{\ast}[Q_1] = [j^{-1}(Q_1)]  = \sum_Q m_Q [Q] \in \CH_0(X_{E_2})$ 
    and $P = j_{E_1/k}(Q_1)$, we have  
    \begin{align*}
      \psi(\{x, j^{\ast}[Q_1]\}_{E_2/k}) 
      &= \sum_Q m_Q \psi(\{x, [Q]\}_{E_2/k}) \\
      &= \sum_Q m_Q \Tr_{E_2(Q)/k(P)}(x) \\
      &= \Tr_{E_1(Q_1)/\kP} \circ \Tr_{E_2/E_1}(x) \\
      &= \psi(\{\Tr_{E_2/E_1}(x), [Q_1]\}_{E_1/k}). 
    \end{align*}
    As a result, the map $\psi$ kills the elements of the form (PF) 
    and this gives $\psi:\Ga \otimesM \cCH_0(X)(k) \to \Coker(\d)$. 
    Next we consider the relation (WR). 
    Let $K = k(C)$ be the function field of a projective smooth curve $C$ over $k$. 
    For $f\in \Kt, g \in K$  and  $\xi$ a closed point in $X_K$, 
    our claim is now 
    \begin{equation}
      \label{eq:claim1}
      \sum_{Q \in C_0}\psi(\{\d_Q(g, f), s_Q([\xi])\}_{k(Q)/k}) = 0.
    \end{equation}
    To show this claim, 
    first we assume 
    $K(\xi) = K$.
    Under this assumption, 
    for each closed point $Q$ in $C$, 
    the valuative criterion for properness gives the diagram 
    \[
      \xymatrix{
                        & \Spec K\ar[d] \ar[r]^-{\xi} & X_K \ar[r] & X\ar[d] \\
      \Spec k(Q) \ar[r] &  \Spec \O_{C,Q} \ar@{-->}[urr]\ar[rr]& & \Spec k. 
      }
    \]
    The composition $\Spec K \onto{\xi} X_K \to X$ 
    extends uniquely to $\pi:C \to X$. 
    Recall that $Q_0$ is the canonical point of $C_{k(Q)}$ 
    determined by $Q$. 
    We denote by $P_0 = \pi_Q(Q_0)$ 
    the closed point of $X_{k(Q)}$ where 
    $\pi_Q:C_{k(Q)} \to X_{k(Q)}$ is the base change of $\pi$. 
    The specialization map is now $s_Q([\xi]) = (\pi_{Q})_{\ast}[Q_0] = [P_0]$ 
    in $\CH_0(X_{k(Q)})$ (\Cf \cite{RS00}, (2.2.2)). 
    If the image of $\xi$ in $X$ 
    gives a closed point $P$ in $X$, 
    then $\pi$ is a constant map;  
    $P = \pi(Q)$ for any $Q\in C_0$. 
    Thus 
    the curve $C$ is defined over $k(P)$ 
    and $P = j_{k(Q)/k}(P_0)$. 
    We obtain 
    \begin{align*}
    &\sum_{Q\in C_0}\psi(\{\d_Q(g, f), s_Q([\xi])\}_{k(Q)/k})\\ 
     &\quad= \sum_{Q\in C_0} \psi(\{\d_Q(g,f), [P_0]\}_{k(Q)/k})\\
     &\quad= \left(\sum_{Q\in C_0} \Tr_{k(Q)/k(P)}\d_Q(g\dlog f)\right) 
    \ \mathrm{at}\ P\in X_0.
    \end{align*}
    The residue theorem for $C$ 
    over $k(P)$ (Thm.\ \ref{thm:res}) 
    gives the claim. 
    On the other hand, 
    when the image $\eta = \pi(\xi)$ of $\xi$ is a point of dimension 1 in $X$, 
    we have 
    \begin{align*}
    &\sum_{Q\in C_0}\psi(\{\d_Q(g, f), s_Q([\xi])\}_{k(Q)/k})\\ 
     &\quad = \sum_{Q\in C_0} \psi(\{\d_Q(g,f), [P_0]\}_{k(Q)/k})\\
     &\quad =  \sum_{Q\in C_0} \left( \Tr_{k(Q)/k(P)}\circ \d_Q(g\dlog f)\ 
       \mathrm{at}\ P=\pi(Q) \in X_0 \right)\\
     &\quad =  \sum_{P\in X_0} \sum_{Q \in \pi^{-1}(P)} 
       \Tr_{k(Q)/k(P)}\circ \d_Q(g \dlog f)\\
     &\quad \stackrel{(\ast)}{=}  \sum_{P\in X_0}\, \dP\circ \Tr_{k(\xi)/k(\eta)}(g\dlog f)\\
     &\quad = \d\circ \Tr_{K/k(\eta)}(g\dlog f)\\ 
     &\quad =0\quad \mathrm{in}\ \Coker(\d). 
    \end{align*}
    Here, the equality ($\ast$) follows from Lemma \ref{lem:trres} below. 

    Next, we show the claim for arbitrary $\xi\in (X_{K})_0$. 
    Then $K' :=K(\xi)$ is a finite extension over $K$. 
    Let $C'$ be the regular proper model of $K'$. 
    For each $Q\in C_0$, the following diagram is commutative: 
    \[
    \vcenter{
    \entrymodifiers={!! <0pt, .8ex>+}
      \xymatrix@C=15mm{
      \CH_0(X_{K'}) \ar[d]_{(j_{K'/K})_{\ast}}\ar[r]^-{e(Q'/Q) s_{Q'}} &
        \displaystyle{\bigoplus_{Q'\in {C'}_0,\ Q'\mid Q}\CH_0(X_{k(Q')})}
        \ar[d]^-{\sum (j_{k(Q')/k(Q)})_{\ast}}\\ 
      \CH_0(X_K) \ar[r]^{s_Q} & \CH_0(X_{k(Q)}),
    }
    }
    \]
    where $e(Q'/Q)$ is the ramification index of 
    $Q'$ over $Q$ (\cite{Ful98}, Sect.\ 20.3). 
    Denoting by $\xi'$ the closed point in $X_{K'}$ 
    which is determined by $\xi$, 
    we have  
    \[
      s_Q([\xi]) = s_Q \circ (j_{K'/K})_{\ast}([\xi']) 
      = \sum_{Q'\mid Q}e(Q'/Q) j_{k(Q')/k(Q)}\circ s_{Q'}([\xi']).
    \]
    The projection formula (PF) and Lemma \ref{lem:rescan} imply  
    \begin{align*}
    &  \sum_{Q \in C_0}\{\d_Q(g, f), s_Q([\xi])\}_{k(Q)/k}\\ 
    &\quad = \sum_{Q \in C_0}\sum_{Q'\mid Q}e(Q'/Q) \{\d_Q(g,f), (j_{k(Q')/k(Q)})_{\ast}s_{Q'}([\xi'])\}_{k(Q)/k}\\
    &\quad = \sum_{Q' \in C'_0}  \{e(Q'/Q)\d_Q(g,f), s_{Q'}([\xi'])\}_{k(Q')/k}\\
    &\quad = \sum_{Q' \in C'_0}  \{\d_{Q'}(g,f), s_{Q'}([\xi'])\}_{k(Q')/k}.
    \end{align*}
    Therefore, the claim is now reduced to the case 
    treated before.

    \sn
    \textit{Proof of the bijection}: 
    It is easy to see that $\psi\circ \phi = \Id$ and thus 
    we show that the map $\phi$ is surjective. 
    For any symbol $\{x,[Q]\}_{E/k}$ with  
    a closed point $Q$ of $X_E$, 
    take $y\in E(Q)$ such that $\Tr_{E(Q)/E}(y) = x$. 
    By the projection formula (PF), we have 
    $\{x,[Q]\}_{E/k} = \{y, (j_{E(Q)/E})^{\ast}[Q]\}_{E(Q)/E}$. 
    Hence we reduce to the case $E(Q) = E$. 
    In this case $E(Q) = E$; $Q_0 = Q$,  
    denoting $P := j_{E/k}(Q)$ the closed point in $X$ 
    determined by $j_{E/k}:X_E \to X$, 
    we have $\kP = k(P_0) = E$. 
    Thus the surjectivity follows from 
    $\phi(x) = \{ x, [P_0]\}_{\kP/k} = \{x, [Q]\}_{E/k}$.
\end{proof}

\begin{lemma}
    \label{lem:trres}
    Let $C' \to C$ be a dominant morphism 
    of projective curves over $k$. 
    If we fix a closed point $P$ of $C$,  
    we have the following commutative diagram: 
    \[
    \vcenter{
    \entrymodifiers={!! <0pt, .8ex>+}
      \xymatrix@C=15mm{
        \Omega^1_{k(C')} \ar[d]_{\Tr_{k(C')/k(C)}}\ar[r]^-{\oplus \d_{P'}} 
          & \displaystyle{\bigoplus_{P' \mid P} k(P')}\ar[d]^-{\sum \Tr_{k(P')/k(P)}} \\
        \Omega^1_{k(C)} \ar[r]^{\dP} & \kP,
      }
    }
    \]
    where 
    $P'$ runs through all points in $C'$ lies above $P$, 
    and 
    $\d_P$ and $\d_{P'}$ 
    are the maps defined in 
    {\rm (\ref{eq:dP+})}.  
\end{lemma}
\begin{proof}
    Since the trace maps are transitive (Thm.\ \ref{thm:trace}, (c)), 
    by taking the normalization, we may assume that 
    $C$ and $C'$ are smooth. 
    Considering the completion $k(C')_{P'}$ and $k(C)_P$ 
    we have 
    \[
      \xymatrix@C=15mm{
        \Omega_{k(C')}^1 \ar[r]\ar[d]_{\Tr_{k(C')/k(C)}} & 
          \Omega_{k(C')_{P'}}^1\ar[d]^{\Tr_{k(C')_{P'}/k(C)_P}}\\
        \Omega_{k(C)}^1 \ar[r] & \Omega_{k(C)_P}^1.
      }
    \]
    This carries over to the local situation. 
    Let $k'/k$ be a finite extension and 
    consider the local fields $K'= k'(\!(t')\!)$ and $K = \ktt$. 
    Now $K^{\ur} = k'(\!(t)\!)$ is the maximal unramified extension in $K'$ over $K$. 
    Since any element $\omega \in \Omega_{K^{\ur}}^1$ is 
    written as 
    \[
      \omega = \sum_{j\in \Z}a_{j}t^j + b_{j}t^{j-1}dt.
    \]
    for some 
    $a_{j}\in \Omega_{k'}^1$, 
    $b_{j} \in k'$, 
    we have 
    \[
    \d_K(\Tr_{K^{\ur}/K}(\omega)) = \Tr_{k'/k}(b_{0}) = \Tr_{k'/k}\circ \d_{K^{\ur}}(\omega).
    \]
    Thus we may assume that $k = k'$ and hence $K^{\ur} = K$. 
    We denote by $\ol{k}$ an algebraic closure of $k$. 
    By the very definition, 
    for any $\omega\in \Omega_{K'}^1$, 
    we have $\Res_{K'}(\omega) = \Res_{\ol{k}(\!(t')\!)}(\omega)$ 
    and the trace map is compatible with base change 
    $\Tr_{\kbar(\!(t')\!)/\kbar(\!(t)\!)}(\omega) = \Tr_{k(\!(t')\!)/k(\!(t)\!)}(\omega)$. 
    We further reduce to the case $k = \kbar$. 
    In this case the residue map $\d_K$ and the residue map $\Res_K$ 
    coincide and the required assertion is well-known (\cite{Kunz}, Thm.\ 17.6).
\end{proof}

From Corollary \ref{cor:WS}, 
the ghost map gives the following. 
\begin{corollary}
    Let $X$ be a projective smooth variety over a perfect field $k$ 
    and  $S$ a finite truncation set. 
    If $\Char(k) = 0$, 
    we have 
    \[
      \Coker \left(\d:\bigoplus_{\xi \in X_1}\WS\Omega^{1}_{k(\xi)} \to 
      \bigoplus_{P\in X_0}\WS(k(P)) \right) 
      \isomto K(k;\WS, \cCH_0(X))
    \]
\end{corollary}

We also have the following complex  
\begin{equation}
\label{eq:cpx}
  \bigoplus_{\xi \in X_1}\WSOmega_{k(\xi)}^1 \onlong{\d} 
  \bigoplus_{P\in X_0} \WS(\kP) \onlong{\Tr} \WS(k)
\end{equation}
where $\Tr$ is defined by the trace maps 
$\Tr := \sum_P \Tr_{\kP/k}$. 
The fact $\Tr \circ \d = 0$ is the residue theorem (Thm.\ \ref{thm:res}). 

\begin{corollary}
    \label{cor:SK1+}
    Let $X$ be a projective smooth and geometrically connected curve over a perfect field $k$. 
    If we assume $X(k)\neq \emptyset$, 
    putting $V(X)^+ := \Ker(\Tr)/\Im(\d)$ 
    we obtain 
    \[
    K(k; \Ga, J_X) \simeq V(X)^+.
    \]
\end{corollary}
\begin{proof}
    Under the assumption $X(k)\neq \emptyset$, 
    the assertion follows from the diagram below 
    (see the proof of Lem.\ \ref{lem:mix} (ii)): 
    \[
     \xymatrix{
      0\ar[r]&K(k; \Ga,J_X) \ar[d]\ar[r]& 
        K(k;\Ga,\cCH_0(X))\ar[d]_{\simeq}^{\psi} \ar[r] &
        K(k;\Ga)\ar[d]^{\simeq} \ar[r] & 0\\
      0\ar[r]& V(X)^+ \ar[r] & \Coker(\d) \ar[r]^{\Tr} & \Ga(k) \ar[r] & 0. 
      }
    \]
\end{proof}

From Lemma \ref{lem:mix}, 
we obtain $K(k;\Ga, J_X) = 0$ if $\Char(k)>0$,  
and this gives the following corollary:
\begin{corollary}
    If $\Char(k)>0$, 
    then we have $V(X)^+  = 0$. In other words, 
    the sequence $(\ref{eq:cpx})$ is exact when $S=\{1\}$.
\end{corollary}

\def\cprime{$'$}
\providecommand{\bysame}{\leavevmode\hbox to3em{\hrulefill}\thinspace}
\providecommand{\href}[2]{#2}

\noindent
Toshiro Hiranouchi \\
Department of Mathematics, Graduate School of Science, Hiroshima University\\
1-3-1 Kagamiyama, Higashi-Hiroshima, 739-8526 Japan\\
Email address: {\tt hira@hiroshima-u.ac.jp}

\end{document}